\documentclass[10pt,a4paper,oneside]{article}
\usepackage[utf8]{inputenc}
\usepackage[english]{babel}
\usepackage{amsmath,amsfonts,amssymb,amsopn,amscd,amsthm,mathtools}

\usepackage[yyyymmdd,hhmmss]{datetime}

\usepackage{fancyhdr}
\newtheorem{Thm}{Theorem}[section]
\newtheorem{Lem}[Thm]{Lemma}
\newtheorem{Prop}[Thm]{Proposition}

\newtheorem{Def}[Thm]{Definition}

\newcommand{\NN}{\mathbb{N}}
\newcommand{\ZZ}{\mathbb{Z}}
\newcommand{\RR}{\mathbb{R}}

\newcommand{\Iff}{{\,\Leftrightarrow\,}}

\newcommand{\Cardinality}[1]{{|#1|}}
\newcommand{\Indicator}[1]{{\mathbf{1}_{\EnumSet{#1}}}}
\newcommand{\Complement}[1]{{#1^c}}
\newcommand{\SymDiff}{{\bigtriangleup\,}}
\newcommand{\DescSet}[2]{{\{#1\mid{}#2\}}}
\newcommand{\LRDescSet}[2]{{\left\{#1\,\middle|\,#2\right\}}}
\newcommand{\EnumSet}[1]{{\{#1\}}}
\newcommand{\LREnumSet}[1]{{\left\{#1\right\}}}
\newcommand{\Modulus}[1]{{|#1|}}
\newcommand{\LRModulus}[1]{{\left|#1\right|}}
\newcommand{\BigModulus}[1]{{\Bigl|#1\Bigr|}}
\newcommand{\TensorTimes}{{\otimes}}
\newcommand{\D}{{\mathrm{d}}}
\newcommand{\Leb}{{\mathcal{L}}}

\newcommand{\Dimension}{{d}}
\newcommand{\RRd}{{\RR^\Dimension}}
\newcommand{\RRplus}{\RR_+}
\newcommand{\Dirac}[1]{{\delta_{#1}}}

\newcommand{\EuclideanDistance}{{\theta}}
\newcommand{\DistanceBetween}[2]{{\EuclideanDistance(#1,#2)}}
\newcommand{\Ball}[1]{{B(#1)}}

\newcommand{\Domein}{{\Lambda}}
\newcommand{\BoundedBorelSets}{{\mathcal{B}_b}}

\newcommand{\StateSpace}{{S}}
\newcommand{\SDom}{{\Delta}}

\newcommand{\StateBorelSets}{{\mathcal{S}}}
\newcommand{\ProjectionBoundedBorelSets}{{\mathcal{S}_{bp}}}

\newcommand{\Conf}[1]{\Omega_{#1}}
\newcommand{\Real}[1]{\omega_{#1}}
\newcommand{\RealBis}[1]{\omega_{#1}'}
\newcommand{\BC}[1]{\gamma_{#1}}

\newcommand{\SigmaAlgebra}[1]{{\mathcal{F}_{#1}}}
\newcommand{\ProductAlgebra}[2]{{\mathcal{F}_{#1}^{\otimes{}#2}}}

\newcommand{\GilbertGraph}[1]{{\mathcal{G}(#1)}}
\newcommand{\ConnectedIn}[1]{{\,\xleftrightarrow{\text{in }{#1}}\,}}

\newcommand{\SymbolPPLaw}{\mathcal{P}}
\newcommand{\GenericPP}[2]{{\SymbolPPLaw^{#1}_{#2}}}
\newcommand{\Expect}[1]{{\mathbb{E}_{#1}}}

\newcommand{\RadiusMeasure}{{Q}}

\newcommand{\SymbolPoisson}{\text{poi}}
\newcommand{\PoissonPP}[1]{{\GenericPP{\SymbolPoisson}{#1}}}

\newcommand{\PercThreshold}[1]{{\alpha_c(#1)}}
\newcommand{\IntCondRad}[1]{{\rho(#1)}}

\newcommand{\PairCorrelationFunction}{\rho}

\newcommand{\Hamiltonian}[1]{{H_{#1}}}
\newcommand{\SymbolPartFunPoisson}{\mathbf{Z}}
\newcommand{\PartFunPoisson}[1]{{\SymbolPartFunPoisson(#1)}}
\newcommand{\SymbolGibbs}{gb}
\newcommand{\GibbsPPActivity}{{\lambda}}
\newcommand{\GibbsPPRadiusMeasure}{{\RadiusMeasure}}
\newcommand{\GibbsPPParams}{{\GibbsPPActivity,\GibbsPPRadiusMeasure}}
\newcommand{\GibbsPP}[1]{{\GenericPP{\SymbolGibbs(\GibbsPPActivity)}{#1}}}

\newcommand{\GibbsStates}{{\mathbf{G}(\GibbsPPActivity{})}}

\newcommand{\DacPoiPPIntensity}{{\overline{\alpha}}}
\newcommand{\DacPoiPPRadiusMeasure}{{\overline{\RadiusMeasure}}}
\newcommand{\DacPoiPPParams}{{\DacPoiPPIntensity,\DacPoiPPRadiusMeasure}}
\newcommand{\Dacf}[1]{{\GenericPP{\text{dac}}{#1}}}

\newcommand{\ThinPoiPPIntensity}{{\alpha}}
\newcommand{\ThinPoiPPRadiusMeasure}{{\GibbsPPRadiusMeasure}}
\newcommand{\ThinPoiPPParams}{{\ThinPoiPPIntensity,\ThinPoiPPRadiusMeasure}}

\newcommand{\BinaryLinearisation}{B}
\newcommand{\Order}{{\preccurlyeq}}
\newcommand{\LTE}{{\preccurlyeq}}
\newcommand{\LinearizedMeasure}{\mathcal{Q}^\star}

\newcommand{\PointThinProba}[3]{{p_{#1}^{\text{s}}(#2\mid{}#3)}}
\newcommand{\JointThinProba}[3]{{p_{#1}^{\text{j}}(#2\mid{}#3)}}
\newcommand{\CouplingThinOne}[1]{{\GenericPP{\text{thin}}{#1}}}

\newcommand{\XPlusEps}{{X_\varepsilon^+}}

\newcommand{\CouplingThinTwo}[1]{{\GenericPP{\text{thin2}}{#1}}}
\newcommand{\InfluenceZone}{{\Gamma}}
\newcommand{\CouplingDisAgZone}[1]{{\GenericPP{\text{da-zone}}{#1}}}
\newcommand{\CouplingDisAgRec}[1]{{\GenericPP{\text{da-rec}}{#1}}}

\newcommand{\HamiltonianCRCM}[1]{{H^{\scriptscriptstyle{crcm}}_{#1}}}

\newcommand{\HamiltonianQuermass}[1]{{H^{\scriptscriptstyle{quer}}_{#1}}}
\newcommand{\Area}{\operatorname*{Area}}
\newcommand{\Perimeter}{\operatorname*{Per}}
\newcommand{\EulerCaracteristic}{\operatorname*{Euler}}

\newcommand{\Abstract}{
\begin{abstract}
We generalise disagreement percolation to Gibbs point processes of balls with varying radii.
This allows to establish the uniqueness of the Gibbs measure and exponential decay of pair correlations in the low activity regime by comparison with a sub-critical Boolean model.
Applications to the Continuum Random Cluster model and the Quermass-interaction model are presented.
At the core of our proof lies an explicit dependent thinning from a Poisson point process to a dominated Gibbs point process.
\end{abstract}
}

\newcommand{\TitleFull}{Disagreement percolation for Gibbs ball models}

\newcommand{\TitleShort}{Disagreement for ball models}

\newcommand{\AuthorsFull}{
Christoph Hofer-Temmel

\thanks{math@temmel.me}
\hspace{2em}
Pierre Houdebert
\thanks{pierre.houdebert@gmail.com}
}

\newcommand{\AuthorsShort}{Hofer-Temmel \& Houdebert}

\newcommand{\Keywords}{Keywords:
continuum random cluster model,
disagreement percolation,
dependent thinning,
Boolean model,
stochastic domination,
phase transition,
unique Gibbs state,
exponential decay of pair correlation
}

\newcommand{\Head}{
 \maketitle
 \Abstract{}
 \Keywords{}\\\MSC{}
 \tableofcontents
}

\title{\TitleFull{}}
\author{\AuthorsFull{}}
\date{}
\begin{document}

\Head{}

\section{Introduction}
\label{sec_intro}

\par
The class of Gibbs models is a rich class of point processes, where a model is defined through its microscopic properties.
The modern formalism is due to Dobrushin~\cite{Dobrushin__DescriptionOfARandomFieldByMeansOfConditionalProbabilitiesAndConditionsForItsRegularity_TVP_1968},
Lanford and Ruelle~\cite{Lanford_Ruelle__ObservablesAtInfinityAndStatesWithShortRangeCorrelationsInStatisticalMechanics_CMP_1969,
Ruelle__SuperstableInteractionsInClassicalStatisticalMechanics_CMP_1970},
who gave their names to the DLR equations defining the Gibbs states through their conditional probabilities.
A classical question is the question of uniqueness of Gibbs states having the same conditional probabilities.
One expects uniqueness at low activities and non uniqueness, usually referred to as \emph{phase transition}, at large activities.
This was proven for example for the well-known Widom-Rowlinson model~\cite{Ruelle__ExistenceOfAPhaseTransitionInAContinuousClassicalSystem,
Chayes_Chayes_Kotecky__TheAnalysisOfTheWidomRowlinsonModelByStochasticGeometricMethods__CMP_1995}.
But uniqueness at low activities is still unproven for most Gibbs model without finite range interaction, such as the Continuum Random Cluster model~\cite{Dereudre_Houdebert__InfiniteVolumeContinuumRandomClusterModel__EJP_2015}.

\par
The aim of this paper is to show uniqueness of the Gibbs state for a large class of Gibbs interactions.
The method used is a continuum extension of the classical disagreement percolation technique introduced by van~den~Berg and Maes~\cite{VanDenBerg_Maes__DisagreementPercolationInTheStudyOfMarkovFields__AP_1994}.
This technique has been recently used to prove uniqueness in the case of the hard-sphere model~\cite{HoferTemmel__DisagreementPercolationForTheHardSphereModel}.
The present paper generalises this construction to the case of Gibbs point processes of balls.
Two natural restrictions are a stochastic domination of the Gibbs point process by a Poisson point process and a locality assumption about the interaction respecting the geometric structure imposed by the balls.
Those are the only assumptions needed for the uniqueness result of this paper.
In particular, the interaction is not restricted to a pair-interaction and may be of higher-order.

\par
The idea behind disagreement percolation is to construct a coupling, named disagreement coupling, between three point processes on a bounded domain.
Two marginals are the studied Gibbs point process with different boundary conditions.
The third marginal is a dominating Poisson point process.
The key property of this coupling is the control of points differing between the two Gibbs instances by the dominating Poisson point process connecting to the boundary.
Therefore, the Poisson point process controls the influence of the boundary conditions.
Interpreting the Poisson point process as a Boolean percolation model, it follows that in the sub-critical percolation regime, this influence is small.
Hence, we derive the uniqueness of the Gibbs phase for activities lower than the critical percolation threshold of the dominating Poisson point process.
In some cases we derive the exponential decay of the pair correlation, proved as a direct consequence of the exponential decay of connectivity in the sub-critical Boolean model.
Our results apply to several Gibbs models such as the continuum random cluster model~\cite{Dereudre_Houdebert__InfiniteVolumeContinuumRandomClusterModel__EJP_2015} or a simplified Quermass-interaction model~\cite{Kendall_VanLieshout_Baddeley__QuermassInteractionProcessesConditionsForStability__AAP_1999}, as well as every Gibbs model with finite range interaction and dominated by a Poisson point process, such as the Strauss model~\cite{Strauss__AModelForClustering_Miometrika_1975}.

\par
The construction of the disagreement coupling is done by recursion in Section~\ref{sec_proof_dacf_existence} and relies strongly on the measurability in the boundary conditions of a coupling between the Gibbs point process and a dominating Poisson point process.
The classic constructions of dominating couplings~\cite{Preston__SpatialBirthAndDeathProcesses__BIIS_1975,Georgii_Kueneth__StochasticComparisonOfPointRandomFields__JAP_1997} are implicit and do not concern themselves with measurability.
In Section~\ref{sec_dep_thinning} we derive a new coupling, namely a dependent thinning from the dominating Poisson point process with explicitly given thinning probabilities.
The thinning probabilities are expressed in terms of the derivative of the free energy of the Gibbs point process.

\par
This paper is organised as follows.
Section~\ref{sec_prelim} introduces the set-up of the paper.
Section~\ref{sec_results} presents the results: uniqueness of the Gibbs state, existence of the disagreement coupling and exponential decay of correlation.
Section~\ref{sec_applications} discusses applications to different Gibbs models, showing that they satisfy the assumptions of our theorems.
A discussion of possible extension, generalisations, and connections to related methods is in Section~\ref{sec_discussion}.
We give an explicit expression for the thinning probabilities in Section~\ref{sec_dep_thinning} and construct the disagreement coupling in Section~\ref{sec_proof_dacf_existence}.
The remaining sections contain proofs of the other statements.

\section{Preliminaries}
\label{sec_prelim}

\subsection{Space}
\label{sec_space}

Let $\DistanceBetween{x}{y}$ be the Euclidean distance between points $x,y \in\RRd$.
For two Borel sets $\Domein_1$ and $\Domein_2$ of $\RRd$, let $\DistanceBetween{\Domein_1}{ \Domein_2}$ be the usual infimum of the pairwise distances between points in $\Domein_1$ and points in $\Domein_2$.
We abbreviate $\DistanceBetween{x}{\Domein}:=\DistanceBetween{\EnumSet{x}}{\Domein}$.
Let $\Leb$ be the Lebesgue measure on $\RRd$.

\par
Consider the state space $\StateSpace:=\RRd\times\RRplus$.
Let $\StateBorelSets$ be the Borel sets of $\StateSpace$.
Let $\ProjectionBoundedBorelSets$ be those Borel sets of $\StateSpace$ whose projection onto $\RRd$ is a bounded Borel set\footnote{The subscript ``bp'' stands for ``bounded projection''.}.
Let $\Conf{}$ be the set of locally finite points configurations on $\StateSpace$, meaning that for each configuration $\Real{}\in\Conf{}$ and each bounded subset $\Domein$ of $\RRd$, the cardinality of the intersection
$\Cardinality{\Real{} \cap (\Domein\times\RRplus )}$ is finite.
For $\SDom \in \StateBorelSets$, denote by $\Conf{\SDom}$ the set of configurations contained in $\SDom$.
For a configuration $\Real{}$, write $\Real{\SDom}$ for $\Real{}\cap\SDom$.
Let $\SigmaAlgebra{}$ be the Borel $\sigma$-algebra on $\Conf{}$ generated by the counting variables.
For $\SDom\in\StateBorelSets$, consider the sub $\sigma$-algebra $\SigmaAlgebra{\SDom}$ generated by the events
\begin{equation*}
\DescSet{\Real{}\in\Conf{}}{\Real{\SDom}\in{}E}, \ E \in \SigmaAlgebra{}
\,.
\end{equation*}
Let $\BoundedBorelSets$ be the bounded Borel sets\footnote{The subscript ``b'' stands for ``bounded''.} of $\RRd$.
In the case of $\SDom = \Domein\times\RRplus$ with $\Domein\in\BoundedBorelSets$, we abbreviate $\Real{\Domein\times\RRplus}$, $\Conf{\Domein\times\RRplus}$ and $\SigmaAlgebra{\Domein\times\RRplus}$  as $\Real{\Domein}$, $\Conf{\Domein}$ and $\SigmaAlgebra{\Domein}$ respectively.

\par
We write $X:=(x,r)\in\StateSpace$.
The \emph{closed ball} of radius $r$ around $x$ is $\Ball{x,r}$ or $\Ball{X}$.
We write $\Ball{\Real{}}:=\cup_{X \in \Real{}} B(X)$.
We abbreviate $\Real{}\cup\EnumSet{X}$ to $\Real{}\cup{}X$.
A configuration $\Real{}\in\Conf{}$ has an associated \emph{Gilbert graph}
$\GilbertGraph{\Real{}}$ with vertex set $\Real{}$ and an edge between $X,Y\in\Real{}$ whenever $\Ball{X}\cap\Ball{Y}\not=\emptyset$.
We say that $X,Y\in\StateSpace$ are \emph{connected by $\Real{}$}, written $X\ConnectedIn{\Real{}}Y$, whenever there is a path in $\GilbertGraph{\Real{}\cup\EnumSet{X,Y}}$ between $X$ and $Y$.
For a non-empty Borel set $\emptyset \not=\Domein\subseteq\RRd$ and a configuration $\RealBis{}$, we write $\Domein\ConnectedIn{\Real{}}\RealBis{}$, if there exists $x\in\Domein$ and $Y\in\RealBis{}$, such that $(x,0)\ConnectedIn{\Real{}}Y$.
This definition extends naturally to connectedness in $\Real{}$ between two non-empty Borel sets of $\RRd$ or between two non-empty configurations.

\subsection{Point processes}
\label{sec_pp}

\par
This work only considers \emph{simple point process} (short: PP).
It treats a PP as a \emph{locally finite random subset of points} of $\StateSpace{}$ instead of as a random measure.
Hence, the law of a PP $\GenericPP{}{}$ is a probability measure on $\Conf{}$ with the canonical variable $\xi$.
Unless there is ambiguity, we refer to a PP by its law and vice-versa.

\par
The most classical PP is the \emph{Poisson point process}.
In this paper we consider only the special case of the Poisson PP with intensity measure
$\alpha\Leb\TensorTimes\GibbsPPRadiusMeasure$, where $\alpha$ is a positive real number called \emph{(spatial) intensity} and $\RadiusMeasure$ is a probability measure on $\RRplus$, called the \emph{radius measure}.
The law of the Poisson PP\footnote{The superscript ``poi'' stands for ``Poisson''.} denoted by $\PoissonPP{\alpha,\RadiusMeasure}$, and the projection on $\SDom\in\StateBorelSets$ is $\PoissonPP{\SDom, \alpha , \RadiusMeasure}$.
An extensive study of the Poisson PP can be found in~\cite{Daley_VereJones__AnIntroductionToTheTheoryOfPointProcesses_II__Springer_2008}.

\par
The percolation properties of $\PoissonPP{\alpha,\RadiusMeasure}$ play an important role in this work.
A configuration $\Real{} \in \Conf{}$ \emph{percolates} if its Gilbert graph $\GilbertGraph{\Real{}}$ contains at least one unbounded (or infinite) connected component.
This is also called the \emph{Boolean model} of percolation.
For a radius measure $\RadiusMeasure$, let $\PercThreshold{\Dimension,\RadiusMeasure}\in[0,\infty]$ be the \emph{threshold intensity} separating the sub-critical ($\PoissonPP{\alpha,\RadiusMeasure}$--almost-never percolating) and super-critical ($\PoissonPP{\alpha,\RadiusMeasure}$--a.s. percolating) phases.
One of these phases may not exist, if and only if $\PercThreshold{\Dimension,\RadiusMeasure}\in\EnumSet{0,\infty}$.
This is always the case in dimension one~\cite[Thm 3.1]{Meester_Roy__ContinuumPercolation__CTM_1996}.
The average volume of a ball under $\RadiusMeasure$ is a dimension-dependent multiple of
\begin{equation}
\label{eq_average_volume}
 \IntCondRad{\RadiusMeasure}
 :=
 \int_{\RRplus} r^\Dimension\RadiusMeasure(\D{}r)
 \,.
\end{equation}

\begin{Thm}[{\cite{Meester_Roy__ContinuumPercolation__CTM_1996,
Gouere__ExistenceOfSubcriticalRegimesInThePoissonBooleanModelOfContinuumPercolation__AP_2008}}]
\label{thm_percolation_boolean}
For $\Dimension \geq 2$, if $\RadiusMeasure$ satisfies $\IntCondRad{\RadiusMeasure}<\infty$, then there exists a percolation threshold
$\PercThreshold{\Dimension,\RadiusMeasure}\!\in\,]0,\infty[$.
Moreover, for each $\alpha < \PercThreshold{\Dimension,\RadiusMeasure}$ and $\Domein\in\BoundedBorelSets$,
\begin{equation}
\label{eq_boolean_subcrit_connect_vanish}
 \PoissonPP{\alpha,\RadiusMeasure}
 \left(\Domein\ConnectedIn{\xi}\Ball{0,n}^c\right)
 \xrightarrow[n \to \infty]{} 0
 \,.
\end{equation}
Furthermore, if the radii are bounded, i.e., $\RadiusMeasure([0,r_0])=1$ for some finite $r_0$, then the previous quantity decays exponentially fast in the distance.
There exist $\kappa, K$ positive such that, for all Borel sets $\Domein_1, \Domein_2$ of $\RRd$ with $\Domein_1$ contained within a $d$-dimensional unit cube,
\begin{equation}
\label{eq_boolean_subcrit_exp_decay}
 \PoissonPP{\alpha,\RadiusMeasure}
 \left(\Domein_1\ConnectedIn{\xi}\Domein_2\right)
 \leq K \exp(-\kappa\,\DistanceBetween{\Domein_1}{\Domein_2})
 \,.
\end{equation}
\end{Thm}

Concerning the case of unbounded radii, a recent paper~\cite{Ahlberg_Tassion_Teixeira__SharpnessOfThePhaseTransitionForContinuumPercolationInR2_arxiv_2016} establishes polynomial decay for the Boolean model in $\RR^2$ with unbounded radii satisfying some integrability assumption.
In a recent preprint~\cite{DuminilCopin_Raoufi_Tassion__SubcriticalPhaseOfdDimensionalPoissonBooleanPercolationAndItsVacantSet}, the authors prove exponential decay of connectivity in the $\Dimension$-dimensional Poisson Boolean model for radii having exponentially fast decaying tails.

\subsection{Gibbs point processes}
\label{sec_gibbs}

\par
For every $\SDom\in\ProjectionBoundedBorelSets$, let there be a \emph{Hamiltonian} $\Hamiltonian{\SDom}:\Conf{\SDom}\times\Conf{\Complement{\SDom}}\to]-\infty,\infty]$ jointly measurable in both arguments.
We assume that the Hamiltonians are additive in the sense that, for all disjoint $\SDom_1,\SDom_2\in\ProjectionBoundedBorelSets$ and $\Real{}^1\in\Conf{\SDom_1}$, $\Real{}^2\in\Conf{\SDom_2}$ and $\BC{}\in\Conf{\Complement{(\SDom_1\cup\SDom_2)}}$,
\begin{subequations}
\label{eq_hamiltonian_basic}
\begin{equation}
\label{eq_hamiltonian_additivity}
 \Hamiltonian{\SDom_1\cup\SDom_2}(\Real{}^1\cup\Real{}^2\mid{}\BC{})
 =
 \Hamiltonian{\SDom_1}(\Real{}^1\mid{}\BC{}\cup\Real{}^2)
 +
 \Hamiltonian{\SDom_2}(\Real{}^2\mid{}\BC{})
 \,.
\end{equation}
Furthermore, we assume $ \Hamiltonian{\SDom}(\emptyset\mid{}\BC{})=0$, which implies together with~\eqref{eq_hamiltonian_additivity} that, if
$\widetilde{\SDom} \subseteq \SDom \in \ProjectionBoundedBorelSets$,
 $\Real{}\in\Conf{\widetilde{\SDom}}$ and $\BC{} \in \Conf{\SDom ^c}$, then
\begin{equation}
\label{eq_hamiltonian_reduction_set}
 \Hamiltonian{\SDom}(\Real{}\mid{}\BC{})
 =
 \Hamiltonian{\widetilde{\SDom}}(\Real{}\mid{}\BC{})
 \,.
\end{equation}
\end{subequations}
The \emph{specification} of the Gibbs PP\footnote{The superscript ``gb'' stands for ``Gibbs''.} on $\SDom\in\ProjectionBoundedBorelSets$ with \emph{boundary condition} $\BC{}$ and \emph{activity} $\GibbsPPActivity$ is the PP law on $(\Conf{\SDom},\SigmaAlgebra{\SDom})$ given by
\begin{equation}
\label{eq_specification}
 \GibbsPP{\SDom,\BC{}}(\D{}\Real{})
 :=
 \frac%
  {e^{-\Hamiltonian{\SDom}(\Real{}\mid{}\BC{})}}
  {\PartFunPoisson{\GibbsPPActivity,\SDom,\BC{}}}
 \PoissonPP{\SDom,\GibbsPPParams}(\D{}\Real{})
 \,,
\end{equation}
where $\PartFunPoisson{\GibbsPPActivity,\SDom,\BC{}}$ is the \emph{partition function} defined by
\begin{equation}
\label{eq_partfun_poisson}
 \PartFunPoisson{\GibbsPPActivity,\SDom,\BC{}}
 :=
 \int
 e^{-\Hamiltonian{\SDom}(\Real{}\mid{}\BC{})}
 \PoissonPP{\SDom,\GibbsPPParams}(\D{}\Real{})
 \,.
\end{equation}
We omit $\Hamiltonian{.}$ and $\RadiusMeasure$ as subscripts, because we consider them given and fixed.
Within different statements additional conditions are placed on them, though.

\par
A PP $\GenericPP{}{}$ is a \emph{Gibbs state} of the specification~\eqref{eq_specification}, if it fulfils the \emph{DLR equations}.
These demand that, for every $\SDom\in\ProjectionBoundedBorelSets$ and $\GenericPP{}{}(\xi_{\Complement{\SDom}}=.)$--a.s.,
\begin{equation}
\label{eq_dlr}
 \GenericPP{}{}(\xi_{\SDom}=\D{}\Real{}
  \mid{}\xi_{\Complement{\SDom}}=\BC{})
 =
 \GibbsPP{\SDom,\BC{}}(\D{}\Real{})
 \,.
\end{equation}
Write $\GibbsStates{}$ for the Gibbs states of~\eqref{eq_specification}.
Throughout this paper we assume that $\GibbsStates{}$ is non-empty.
The question of existence of Gibbs states is classical and difficult.
Existence proofs often rely on a fine study of the interaction and the question remains open for many models.
For all models considered in Section~\ref{sec_applications}, though, existence has already been established and the corresponding references are given.

\par
We point out that we have defined $\PoissonPP{\alpha,\RadiusMeasure}$ and $\GibbsPP{\SDom,\BC{}}$ as simple PPs on $\StateSpace$ and $\SDom$ respectively, instead of as marked PPs (with the marks being the radii of the balls).
First, this lets us avoid the notational overhead associated with marked PPs.
Second, we use the Euclidean nature of $\StateSpace$, i.e., the fact that the marks lie in $\RRplus$, in Section~\ref{sec_ordering}.
See also the discussion in Section~\ref{sec_discussion}.

\subsection{Stochastic domination}
\label{sec_domination}

\par
Let $\SDom\in\BoundedBorelSets$.
On $\Conf{\SDom}^n$, the standard product $\sigma$-algebra is $\ProductAlgebra{\SDom}{n}$.
The canonical variables on $\Conf{\SDom}^n$ are $\xi:=(\xi^1,\dotsc,\xi^n)$.
A \emph{coupling} $\GenericPP{}{}$ of $n$ PP laws $\GenericPP{1}{},\dotsc,\GenericPP{n}{}$ on $\SDom$ is a probability measure on $(\Conf{\SDom}^n,\ProductAlgebra{\SDom}{n})$ such that, for all $1\le{}i\le{}n$ and $E\in\SigmaAlgebra{\SDom}$, $\GenericPP{}{}(\xi^i\in E)=\GenericPP{i}{}(\xi\in E)$.

\par
An event $E \in \SigmaAlgebra{}$ is called \emph{increasing}, if $\Real{} \in E$ implies that $\Real{} \cup{}\RealBis{}\in E$, for every $\RealBis{}\in\Conf{}$.
If $\GenericPP{1}{}$ and $\GenericPP{2}{}$ are two probability measures, then we say that $\GenericPP{2}{}$ \emph{stochastically dominates} $\GenericPP{1}{}$ (short: dominates), if
$\GenericPP{1}{} (E) \leq \GenericPP{2}{}(E)$ for all increasing events $E$.
By Strassen's theorem~\cite{Liggett__InteractingParticleSystems__FPMS_1985}, this is equivalent to the existence of a coupling $\GenericPP{}{}$ of $\GenericPP{1}{}$ and $\GenericPP{2}{}$ such that $\GenericPP{}{}(\xi^1 \subseteq \xi^2)=1$.
In the context of PPs, this is also called a \emph{thinning} from $\GenericPP{2}{}$ to $\GenericPP{1}{}$, alluding to the a.s. removal of points from the dominating PP realisation to get a realisation of the dominated PP.

\par
The \emph{Papangelou intensity} is the conditional intensity of adding a point at $X$ to $\Real{}\in\Conf{}$~\cite[Section 15.5]{Daley_VereJones__AnIntroductionToTheTheoryOfPointProcesses_II__Springer_2008}.
Abbreviating $\EnumSet{X}$ to $X$, this is the quantity
\begin{equation*}
\label{eq_papangelou}
 \lambda\exp(-\Hamiltonian{X}(X \mid{} \Real{}))
 \,.
\end{equation*}
A classic sufficient condition for stochastic domination of a Gibbs PP of activity $\GibbsPPActivity$ by a Poisson PP of intensity $\alpha$ is the uniform boundedness of the Papangelou intensity~\cite{Preston__SpatialBirthAndDeathProcesses__BIIS_1975,Georgii_Kueneth__StochasticComparisonOfPointRandomFields__JAP_1997}.
That is,
\begin{equation}
\label{eq_papangelou_bounded_uniform}
\tag{Dom}
 \lambda\exp(-\Hamiltonian{X}(X \mid{} \Real{}))
 \le
 \alpha
 \,.
\end{equation}
This is equivalent to a uniform lower bound on the local energy
$\Hamiltonian{X}(X \mid{} \Real{})$.

\section{Results and discussion}
\label{sec_results}

\subsection{The method of disagreement percolation}
\label{sec_disagreement}

The idea behind disagreement percolation is to couple two instances of a Gibbs PP on the same $\SDom\in\ProjectionBoundedBorelSets$ with arbitrary boundary conditions, such that the set of points differing between the two instances (the \emph{disagreement cluster}) is dominated by a Poisson PP.
Let $\SymDiff$ be the symmetric set difference operator.

\begin{Def}
\label{def_dacf}
A \emph{disagreement coupling\footnote{The superscript ``dac'' stands for ``disagreement coupling''.} family at level} $(\DacPoiPPParams{})$ is a family of couplings $(\Dacf{\SDom,\BC{}^1,\BC{}^2})$ indexed by $\SDom\in\ProjectionBoundedBorelSets$ and
$\BC{}^1,\BC{}^2 \in \Conf{\Complement{\SDom}}$ measurable in the boundary conditions and fulfilling
\begin{subequations}
\label{eq_dacf_properties}
\begin{align}
\label{eq_dacf_gibbs}
 \forall 1\le{}i\le{}2:\quad
 &\Dacf{\SDom,\BC{}^1,\BC{}^2}(\xi^i=\D{} \Real{})
  =
  \GibbsPP{\SDom,\BC{}^i}(\D{}\Real{})
 \,,\\
 \label{eq_dacf_poi}
 &\Dacf{\SDom,\BC{}^1,\BC{}^2}(\xi^3=\D{}\Real{})
  =
  \PoissonPP{\SDom,\DacPoiPPParams{}}(\D{}\Real{})
 \,,\\
 \label{eq_dacf_domination}
 &\Dacf{\SDom,\BC{}^1,\BC{}^2}(\xi^1\cup{}\xi^2\subseteq\xi^3)=1
\intertext{and}
 \label{eq_dacf_connected}
 &\Dacf{\SDom,\BC{}^1,\BC{}^2}
   (\forall X \in \xi^1 \SymDiff{}\xi^2:
    X\ConnectedIn{\xi^3}\BC{}^1\cup{}\BC{}^2
   ) = 1
 \,.
\end{align}
\end{subequations}
\end{Def}

The first three properties~\eqref{eq_dacf_gibbs},~\eqref{eq_dacf_poi} and~\eqref{eq_dacf_domination} describe a joint thinning from the Poisson PP to two Gibbs PPs, each with its own boundary condition.
The final property~\eqref{eq_dacf_connected} controls the influence of the boundary conditions on difference between the Gibbs PP realisations and is non-trivial.

\par
To enable~\eqref{eq_dacf_connected}, the Hamiltonian $\Hamiltonian{\SDom}(\Real{}\mid{}\BC{})$ should depend only on those points in $\BC{}$ which are connected to $\Real{}$ in $\GilbertGraph{\Real{}\cup\BC{}}$.
With~\eqref{eq_hamiltonian_additivity}, this is equivalent to the following.
If $\Real{} \in \Conf{\SDom}$ and $\BC{} \in \Conf{\SDom ^c}$ are such that $\Real{}$ and $\BC{}$ are not connected in $\GilbertGraph{\Real{} \cup \BC{}}$, then
\begin{equation}
\label{eq_hamiltonian_locality}
\tag{Loc}
 \Hamiltonian{\SDom}(\Real{}\mid{}\BC{})
 =
 \Hamiltonian{\SDom}(\Real{}\mid{}\emptyset)
 \,.
\end{equation}
As stated in the following theorem, a disagreement coupling family exists under natural conditions.

\begin{Thm}
\label{thm_dacf_existence}
If conditions~\eqref{eq_papangelou_bounded_uniform} and~\eqref{eq_hamiltonian_locality} are fulfilled,
then there exists a disagreement coupling family at level $(\ThinPoiPPParams)$.
\end{Thm}

The proof of Theorem~\ref{thm_dacf_existence} is in Section~\ref{sec_proof_dacf_existence}.
The key property~\eqref{eq_dacf_connected} joined with~\eqref{eq_dacf_domination} allows control of the disagreement cluster $\xi^1 \SymDiff \xi^2$ by a percolation cluster of a Boolean model connected to the boundary.
Hence, a sub-critical Boolean model implies the uniqueness of the Gibbs state.

\begin{Thm}
\label{thm_unique_gibbs}
If there exists a disagreement coupling family at level $(\DacPoiPPParams{})$ such that
$\DacPoiPPRadiusMeasure$ satisfies the integrability assumption
$\IntCondRad{\DacPoiPPRadiusMeasure}<\infty$ and is sub-critical with
$\DacPoiPPIntensity < \PercThreshold{\Dimension,\DacPoiPPRadiusMeasure}$, then there is a unique Gibbs state in $\GibbsStates$.
\end{Thm}

In the case of bounded radii, the connection probabilities decay exponentially in a sub-critical Boolean model.
This translates into an exponential decay in influence of the boundary condition of the Gibbs PP and the reduced pair correlation function of the Gibbs state.

\begin{Thm}
\label{thm_correlations}
Assume that $\Dacf{}$ is a disagreement coupling family for $\GibbsPP{}$ at level
$(\DacPoiPPParams)$ such that $\DacPoiPPRadiusMeasure$ has bounded support and such that
$\DacPoiPPIntensity < \PercThreshold{\Dimension,\DacPoiPPRadiusMeasure}$.
Let $\kappa$ be the constant from~\eqref{eq_boolean_subcrit_exp_decay} for $\PoissonPP{\DacPoiPPParams}$.
Then, $\GenericPP{}{}$ is the unique Gibbs state in $\GibbsStates$ and there exists $K'>0$ such that:
\begin{subequations}
\label{eq_correlations}
\newline
For all $\Domein_1,\Domein_2\in\BoundedBorelSets$ with $\Domein_1\subseteq\Domein_2$ and $\Domein_1$ contained within a $d$-dimensional unit cube, $\BC{} \in \Conf{\Domein_2^c}$ and $E \in \SigmaAlgebra{\Domein_1}$,
\begin{equation}
\label{eq_correlations_spec}
 \Modulus{
  \GibbsPP{\Domein_2, \BC{}}(\xi_{\Domein_1}\in{}E)
  -
  \GenericPP{}{}(\xi_{\Domein_1}\in{}E)
 }
 \leq K' \exp(-\kappa\,\DistanceBetween{\Domein_1}{\Domein_2^c})
 \,.
\end{equation}
For all $\Domein_1,\Domein_2\in\BoundedBorelSets$ with $\Domein_1$ contained within a $d$-dimensional unit cube, $E\in\SigmaAlgebra{\Domein_1}$ and $F\in\SigmaAlgebra{\Domein_2}$,
\begin{equation}
\label{eq_correlation_event}
 \Modulus{
  \GenericPP{}{}(E\cap{}F)
  -
  \GenericPP{}{}(E)\GenericPP{}{}(F)
 }
 \leq K' \exp(-\kappa\,\DistanceBetween{\Domein_1}{\Domein_2})
 \,.
\end{equation}
Finally, for all $x,y\in \RR^\Dimension$, the pair correlation function $\PairCorrelationFunction (x,y)$ decays exponentially as
\begin{equation}
\label{eq_correlation_pair}
\PairCorrelationFunction (x,y)
 \leq K' \exp(-\kappa\,\DistanceBetween{x}{y})
 \,.
\end{equation}
\end{subequations}
\end{Thm}
The proofs of Theorem~\ref{thm_unique_gibbs} and Theorem~\ref{thm_correlations} are in Section~\ref{sec_dap_proofs}.
If one wants to consider general $\Domein_1$ in~\eqref{eq_correlations_spec} and ~\eqref{eq_correlation_event}, then a union bound yields an upper bound multiplied by the number of $d$-dimensional unit cubes needed to cover $\Domein_1$.

\subsection{Applications}
\label{sec_applications}

This section applies the theorems from Section~\ref{sec_disagreement} to several classical Gibbs models, such as the Continuum random cluster model and the Quermass-interaction model.
To the best of our knowledge, this is the first time that uniqueness at low activities is proved for these models.

\subsubsection{Gibbs models with finite range interaction}
\label{sec_finite_range}

\par
Consider Gibbs models on $\RRd$ with finite range interaction $R>0$.
Examples of such models are the hard-sphere model, the area interaction model with deterministic radii or the Strauss model~\cite{Strauss__AModelForClustering_Miometrika_1975}.
A general result of Preston~\cite{Preston__RandomFields__LNM_Springer_1976} establishes the existence of a Gibbs state, hence $\GibbsStates$ is never empty.
By taking $\RadiusMeasure = \Dirac{R}$, these models fit the setting of the present article and the condition~\eqref{eq_hamiltonian_locality} is automatically fulfilled.
If the model satisfies condition~\eqref{eq_papangelou_bounded_uniform}, as it is the case with the Strauss model, then by applying  Theorem~\ref{thm_dacf_existence}, Theorem~\ref{thm_unique_gibbs} and Theorem~\ref{thm_correlations}, we obtain the uniqueness of the Gibbs state and the exponential decay of pair correlations at low activity.
In the case of the hard-sphere model, this result was already proved in~\cite{HoferTemmel__DisagreementPercolationForTheHardSphereModel}.

\subsubsection{Continuum random cluster model}
\label{sec_model_crcm}
\par
The continuum random cluster model, also known as continuum FK-percolation model, is a Gibbs model of random balls whose interaction depends on the number of connected components of the Gilbert graph.
This model, introduced in the 1980s as a continuum analogue of the well-known (lattice) random cluster model~\cite{Grimmet__TheRandomClusterModel__FPMS_2006}, was the original motivation for doing the present article.
Recently, existence and percolation properties of this model were investigated in~\cite{Dereudre_Houdebert__InfiniteVolumeContinuumRandomClusterModel__EJP_2015,
Houdebert_PercolationResultsForTheContinuumRandomClusterModel}.
Formally, for $X\in\StateSpace$ and $\Real{}\in\Conf{}$, we have
\begin{equation*}
 \HamiltonianCRCM{X}(X \mid{} \Real{})
 :=
 - \log(q)\left(1-k(X,\Real{})\right)
 \,,
\end{equation*}
where $q>0$ is the connectivity parameter of the model and $k(X,\Real{})$ denotes the number of connected components of $\GilbertGraph{\Real{}}$ connected to $X$ in $\GilbertGraph{\Real{}\cup X}$.
This model, even in the case of bounded radii, is not finite range as the function $k$ may depend on the configuration $\Real{}$ arbitrary far from the added point $X$.
The dependence of $\HamiltonianCRCM{X}(X \mid{} \Real{})$ on $\Real{}$ via $k(X,\Real{})$ implies that the model satisfies~\eqref{eq_hamiltonian_locality}.
It also satisfies~\eqref{eq_papangelou_bounded_uniform}, because
\begin{equation*}
 \HamiltonianCRCM{X}(X \mid{} \Real{})
 \geq -\log q
 \,.
\end{equation*}
By Theorem~\ref{thm_dacf_existence}, if $\RadiusMeasure$ satisfies
$\IntCondRad{\RadiusMeasure} < \infty$, then there exists a disagreement coupling family at level $(\lambda{}q,\RadiusMeasure)$.
So by Theorem~\ref{thm_unique_gibbs}, if $\lambda < \PercThreshold{d,\RadiusMeasure}/q$, there is a unique Gibbs state.

\subsubsection{Quermass-interaction model}
\label{sec_quermass}
\par
The Quermass-interaction model is a Gibbs model of random balls in $\RR^2$ whose interaction depends on the perimeter, area and Euler characteristic of the random structure.
It was introduced in~\cite{Kendall_VanLieshout_Baddeley__QuermassInteractionProcessesConditionsForStability__AAP_1999}.
The existence of the infinite volume Gibbs model has been proven in
\cite{Dereudre__TheExistenceOfQuermassInteractionProcessesForNonlocallyStableInteractionAndNonboundedConvexGrains__AAP_2009} and the existence of a supercritical percolation phase has been shown in
\cite{Coupier_Dereudre__ContinuumPercolationForQuermassInteractionModel__EJP_2014}.

Fix $\theta_1,\theta_2,\theta_3\in\RR$.
Let $\Area(X,\Real{})$, $\Perimeter(X,\Real{})$ and $\EulerCaracteristic(X,\Real{})$ be the variation of the area, the perimeter and the Euler characteristic respectively, when the ball $\Ball{X}$ is added to $\Ball{\Real{}}$.
The local energy of the Quermass-interaction model is
\begin{equation*}
 \HamiltonianQuermass{X}(X \mid{} \Real{} )
 :=
 \theta_1\Area(X,\Real{})
 + \theta_2\Perimeter(X,\Real{})
 + \theta_3\EulerCaracteristic(X,\Real{})
 \,.
\end{equation*}
\par
The contribution of the Euler characteristic is difficult to control.
In particular when $\theta_3 \not = 0$, the domination condition~\eqref{eq_papangelou_bounded_uniform} is not satisfied, even with deterministic radii.

From here on, we only consider the case of $\theta_3 = 0$ and the radius having support on some positive finite interval, meaning that $ \RadiusMeasure ([r_0,r_1])=1 $ for some $0<r_0 \leq r_1<\infty \,$.
In this setting, the interaction is local and satisfies~\eqref{eq_hamiltonian_locality}.
Using standard bounds from~\cite[Lemma 4.12]{Coupier_Dereudre__ContinuumPercolationForQuermassInteractionModel__EJP_2014}, we have upper and lower bounds for the Papangelou intensity and~\eqref{eq_papangelou_bounded_uniform} is satisfied.
Therefore, by applying Theorem~\ref{thm_dacf_existence}, Theorem~\ref{thm_unique_gibbs} and Theorem~\ref{thm_correlations}, one gets the uniqueness of the Quermass-interaction Gibbs phase and the exponential decay of pair correlations for small enough the activity $\GibbsPPActivity$, depending on the parameters $\PercThreshold{2,Q}, \theta_1, \theta_2, r_0$ and $r_1$.

\subsubsection{Widom-Rowlinson model with random radii}
\label{sec_model_wr}
\par
The Widom-Rowlinson model is a well-known model of statistical mechanics introduced in 1970~\cite{Widom_Rowlinson__NewModelForTheStudyOfLiquidVaporPhaseTransitions__JCP_1970} originally to model the interaction between two gases.
It is also the first continuum model for which a phase transition was proved, by Ruelle~\cite{Ruelle__ExistenceOfAPhaseTransitionInAContinuousClassicalSystem} using Peierl's argument.
A modern proof of this phase transition was done by Chayes, Chayes and Koteck\'y~\cite{Chayes_Chayes_Kotecky__TheAnalysisOfTheWidomRowlinsonModelByStochasticGeometricMethods__CMP_1995} using percolation properties of the continuum random cluster model.
We consider here the generalised model with random and possibly unbounded radii.
\par
This model does not follow strictly the setting of the article, because each ball is assigned a colour mark $i$ belonging to some finite set of cardinality $q$.
The Hamiltonian is a hard-core constraint on the colouring.
Configurations with overlapping balls of different colours are forbidden.
In other words, each connected component of the Gilbert graph must be mono-coloured.
\par
To apply our result we are using a well known representation of the Widom-Rowlinson model: the \emph{Fortuin-Kasteleyn} representation.
In the infinite volume case it states~\cite[Proposition 3.1]{Houdebert_PercolationResultsForTheContinuumRandomClusterModel} that a Widom-Rowlinson measure with at most one infinite connected component is a continuum random cluster model with a uniform independent colouring of the finite connected component.
But, for small enough activities, the Widom-Rowlinson measures are not percolating, whence the Fortuin-Kasteleyn representation yields in this case a bijection between Widom-Rowlinson and continuum random cluster Gibbs phases, proving the wanted uniqueness result.

\subsection{Discussion}
\label{sec_discussion}

The domination condition~\eqref{eq_papangelou_bounded_uniform} can sometimes be very restrictive.
One way to weaken this condition is to demand the following bound on the local energy.
\begin{equation}
\label{eq_local_energy_domination_generalize}
\tag{Weak-Dom}
 \Hamiltonian{\EnumSet{(x,r)}}((x,r) \mid{} \Real{})
 \geq g(r)
 \,,
\end{equation}
where $g$ is a nice enough measurable function.
In that case the dominating Poisson PP has a radius measure with density proportional to $e^{-g(r)}$ with respect to $\GibbsPPRadiusMeasure$.
However the construction of the dependent thinning done in Section~\ref{sec_dep_thinning} does not carry over without new difficult conditions.
If one gets the existence of a coupling with measurability with respect to the boundary condition, then the construction of the disagreement coupling family done in Section~\ref{sec_proof_dacf_existence} would carry over unchanged and Theorem~\ref{thm_dacf_existence} would still be valid.

\par
As mentioned in Section~\ref{sec_model_wr},
The Widom-Rowlinson model does not fit the setting of the present article,
 and uniqueness at low activity is derived from the standard Fortuin-Kasteleyn representation and from the uniqueness at low activity of the CRCM.
However the construction of the dependent
thinning in Section~\ref{sec_dep_thinning} and the construction of the disagreement coupling in Section~\ref{sec_dac_construction} carry over to the more general case of Gibbs model of random balls, satisfying conditions~\eqref{eq_hamiltonian_locality} and~\eqref{eq_papangelou_bounded_uniform}, with an extra set
of marks (representing for instance type, temperature, energy,...) in $\RR^p$.

\par
For difficult radius measures, the analysis of the dominating Boolean model might be complicated.
If a radius measure $\RadiusMeasure'$ dominates $\DacPoiPPRadiusMeasure$, then the gluing lemma~\cite[Chapter 1]{Villani__OptimalTransport_OldAndNew__Springer_2009} couples the disagreement coupling with a dominating coupling between $\PoissonPP{\SDom,\DacPoiPPParams}$ and $\PoissonPP{\SDom,\DacPoiPPIntensity,\RadiusMeasure'}$.
This allows to use the condition $\DacPoiPPIntensity<\PercThreshold{\Dimension,\RadiusMeasure'}$ in Theorem~\ref{thm_unique_gibbs}.

\par
If the mark of a point is not a ball, but a different geometric object of a fixed shape such as a square, then the approach and results should generalise by a) comparing with a Boolean percolation model of squares or b) including the shape in a larger ball and comparing with a Boolean percolation model on such balls.
An example of where such a reasoning would apply is the segment process in~\cite{Benes_Novotna__CentralLimitTheoremForFunctionalsOfGibbsParticleProcesses}.

\par
If the geometric objects are described by more real parameters and the objects are monotone growing in the parameters, a straightforward extension of the derivation approach could work, too.
If the marks are more general, such as compact sets, and their distribution is such that they can be included in larger balls with some radius law, then again a comparison with a Boolean percolation model of this radius law yields uniqueness and exponential decay of the pair correlation.
For even more general mark spaces and measures, a possible approach could be to split the derivation in Section~\ref{sec_dep_thinning} into a purely spatial component and work with the joint mark distributions conditional on the locations.

\par
The coupling constructed in Theorem~\ref{thm_dacf_existence} reduces to the coupling family for the hard-sphere model used in~\cite{HoferTemmel__DisagreementPercolationForTheHardSphereModel}.
For other finite-range Gibbs models, it improves upon the conjectured general product construction discussed in~\cite{HoferTemmel__DisagreementPercolationForTheHardSphereModel} by a factor of two.
See the discussion at the beginning of Section~\ref{sec_proof_dacf_existence}.

\par
Other classic conditions for uniqueness of the low-activity Gibbs measure are cluster expansion and Dobrushin uniqueness.
An explicit comparison with cluster expansion has been done for the hard-sphere model in~\cite{HoferTemmel__DisagreementPercolationForTheHardSphereModel}.
It shows that disagreement percolation is better in dimensions one to three and suggests a way to show the same for high dimensions.
If one derives the exponential decay of all correlations as in Theorem~\ref{thm_correlations}, then complete analyticity would follow~\cite{Dobrushin_Shlosman__CompletelyAnalyticalInteractions_AConstructiveDescription__JStatPhys_1987}, too.
Dobrushin uniqueness~\cite{Dobrushin__DescriptionOfARandomFieldByMeansOfConditionalProbabilitiesAndConditionsForItsRegularity_TVP_1968}, generalised to finite-range interaction Gibbs PPs in~\cite[Thm 2.2]{Klein__DobrushinUniquenessTechniquesAndTheDecayOfCorrelationsInContinuumStatisticalMechanics__CMP_1982}, derives uniqueness from the summability of the variation distance between two Gibbs instances with the same boundary condition, except on a finite set of points.
In the setting of our paper, the Dobrushin condition can be checked using the disagreement coupling and the exponential bounds from Theorem~\ref{thm_correlations}.

\section{Proof of Theorem~\ref{thm_dacf_existence}}
\label{sec_proof_dacf_existence}

\par
The main result of this section is the construction of a disagreement coupling family in Section~\ref{sec_dac_construction}.
A key building block is a dependent thinning from Poisson PP and Gibbs PP $\GibbsPP{\SDom,\BC{}}$ in Section~\ref{sec_dep_thinning}.
This dependent thinning is measurable in the boundary condition and a key building block in the recursive construction of the disagreement coupling family.
The dependent thinning comes from an ordered exploration of the domain keeping or thinning points depending on the not yet explored space and the already kept points.
The single point thinning probability is explicitly given as a derivative of the free energy of the yet unexplored part of the domain.

\par
While the general approach mirrors the one taken in~\cite{HoferTemmel__DisagreementPercolationForTheHardSphereModel}, there are improvements beyond the generalisation to Gibbs PPs with more general interactions.
First, the single point thinning probability is not guessed and the proven to be the correct quantity, but derived by a principled approach by differentiating certain void probabilities.
Second, the disagreement coupling in Section~\ref{sec_dac_construction} does away with hard-sphere specific details and gets $\ThinPoiPPIntensity$ as intensity of the dominating Poisson PP.
This improves upon the conjectured general product construction discussed in~\cite{HoferTemmel__DisagreementPercolationForTheHardSphereModel} by a factor of two and seems to be optimal as discussed in that paper.
Finally, the construction of~\cite{HoferTemmel__DisagreementPercolationForTheHardSphereModel} is improved by considering balls with random radii.
In order to handle the case of unbounded radii, a control with high probability of the radii of the Poisson Boolean model in Lemma~\ref{lem_disconnect} is used to prove our uniqueness result Theorem~\ref{thm_unique_gibbs}.

\subsection{Measurable ordering}
\label{sec_ordering}

This section presents a measurable total ordering of $\SDom\in\BoundedBorelSets$ and a notion of a one-sided Lebesgue derivative in~\eqref{eq_derivative_def}.
\par
First, map a non-negative real number to its shortest binary digit expansion, filled up with zeros to a bi-infinite sequence of $0$s and $1$s.
In the case of $x$ being a multiple of $2^n$, for some $n\in{}\ZZ$, this avoids the representation of $x$ with only $1$s below index $n$.
For example, with $\bar{a}$ denoting an infinite sequence of the digit $a\in\EnumSet{0,1}$ and the decimal point ``$.$'' to the left of the power $0$ coefficient, $2$ maps to $\bar{0}10.\bar{0}$ instead of $\bar{0}1.\bar{1}$.
With
\begin{equation*}
 D:=
 \DescSet{\iota\in\EnumSet{0,1}^\ZZ}%
 {\forall{}n\in\ZZ:\exists{}m\le{}n\in\ZZ: \iota_m=0
  \,,
  \exists{}m\in\ZZ:\forall{}n\ge{}m:\iota_n=0
 }
\end{equation*}
the mapping is
\begin{equation}
\label{eq_binary_digits}
 b:
 \qquad
 \RRplus\to{}D
 \qquad
 x:=\sum_{n\in\ZZ} \iota_n 2^n \mapsto (\iota_n)_{n\in\ZZ}
 \,.
\end{equation}
All sequences in $D$ with a fixed prefix up to $k$ positions left to the decimal point, and only those, are mapped to the same interval of length $2^{-k}$ in $\RRplus$.
By the above discussion the map $b$ is bijective and measurable in both directions.

\par
Second, we use $b$ to linearise $\RRplus^m$.
Consider the map
\begin{equation}
\label{eq_binary_linearisation}
 \BinaryLinearisation:
 \qquad
 \RRplus^m\to\RRplus
 \qquad
 x\mapsto
 b^{-1}(n\mapsto b(x_{n\operatorname{mod}m})_{\lfloor{}n/m\rfloor{}})
 \,.
\end{equation}
The map $\BinaryLinearisation$ juxtaposes the digits of the same power of $2$ of the coordinates of $x$ and maps the result back to $\RRplus$.
The map $\BinaryLinearisation$ is bijective and measurable in both directions.
It equips $\RRplus^m$ with a measurable total order $\Order$ defined by
\begin{equation}
\label{eq_order}
 x\LTE{}y \Iff{} \BinaryLinearisation(x)\le{}\BinaryLinearisation(y)
 \,.
\end{equation}
For $n\in\ZZ$ and $\vec{a}\in\NN^m$, the hyperblock $\prod_{i=1}^m [\frac{a_i}{2^n},\frac{a_i+1}{2^n}[\,\subseteq\RRplus^m$ and the interval $[\BinaryLinearisation(\vec{a}),\BinaryLinearisation(\vec{a})+\frac{1}{2^{nm}}[\,\subseteq\RRplus$ are in bijection.

\par
Apply the ordering and bijection from above to $\RRplus^{\Dimension+1}$.
Without loss of generality, translation allows to apply the ordering and bijection to every $\SDom\in\ProjectionBoundedBorelSets$, by shifting its support into $\RRplus^\Dimension$.
Although $\GibbsPPRadiusMeasure$ may contain atoms, $\Leb\TensorTimes\GibbsPPRadiusMeasure$ and $\LinearizedMeasure:=(\Leb\TensorTimes\GibbsPPRadiusMeasure)\circ\BinaryLinearisation^{-1}$ are diffuse.
Because $\BinaryLinearisation$ is a measurable bijection, we may not always write it and switch between $\RRplus^{\Dimension+1}$ and its linearisation in a notational lightweight and implicit fashion.
This also holds for the measures above.

\par
For $\LinearizedMeasure$-a.e. $X\in\SDom$, there exists $\varepsilon>0$ and $\XPlusEps\in\SDom$ with $X\LTE\XPlusEps$ such that $\LinearizedMeasure([X,\XPlusEps])=\varepsilon$.
For a function $f:\SDom\to\RR$, we define a one-sided variant of the Lebesgue-derivative at $X$ by
\begin{equation}
\label{eq_derivative_def}
 \frac{\partial{}f}{\partial{}X}(X)
 :=
 \lim_{\varepsilon\to{}0} \frac{f(\XPlusEps)-f(X)}{\varepsilon}
 \,,
\end{equation}
whenever this limit exists.

\subsection{The dependent thinning}
\label{sec_dep_thinning}

Given $\SDom\in\ProjectionBoundedBorelSets$ and $\BC{}\in\Conf{\Complement{\SDom}}$, we want to thin $\PoissonPP{\SDom,\ThinPoiPPParams}$  to $\GibbsPP{\SDom,\BC{}}$.
We order $\SDom$ by $\Order$ and restrict intervals $]X,Y]$ (and all variants thereof) to $\SDom$.
\begin{Prop}
\label{prop_thinning}
Under assumption~\eqref{eq_papangelou_bounded_uniform}, a thinning\footnote{The superscript ``thin'' stands for ``thinning''.} from $\PoissonPP{\SDom,\ThinPoiPPParams}$ to $\GibbsPP{\SDom,\BC{}}$ is given by
\begin{equation}
\label{eq_thinning}
 \CouplingThinOne{\SDom,\BC{},\ThinPoiPPParams}
 (\D{}(\Real{}^1,\Real{}^2))
 :=
 \JointThinProba{\SDom,\BC{}}{\Real{}^1}{\Real{}^2}
 \PoissonPP{\SDom,\ThinPoiPPParams}(\D{}\Real{}^2)
 \,,
\end{equation}
with the indicator function being $\Indicator{.}$ and the joint thinning
probability
\footnote{The superscript ``j'' stands for ``joint''.}
$\JointThinProba{\SDom,\BC{}}{\RealBis{} }{ \Real{} }$ being
\begin{multline}
\label{eq_thinprobajoint_def}
 \JointThinProba{\SDom,\BC{}}{\RealBis{} }{ \Real{} }
 :=
 \Indicator{\RealBis{} \subseteq{} \Real{} }
 \left(
  \prod_{Y \in{} \RealBis{}}
  \PointThinProba{\SDom,\BC{}}{Y}{ \RealBis{]-\infty,Y[}}
 \right)
 \\\times
 \left(
  \prod_{Z \in{} \Real{} \setminus{} \RealBis{} }
  \left(1-\PointThinProba{\SDom,\BC{}}{Z}{\RealBis{]-\infty,Z[}}\right)
 \right)
 \,,
\end{multline}
and the dependent single point thinning probability
\footnote{The superscript ``s'' stands for ``single point''.}
$\PointThinProba{\SDom,\BC{}}{X}{\RealBis{}}$ being
\begin{align}
\label{eq_thinproba_logarithm}
 \PointThinProba{\SDom,\BC{}}{X}{\RealBis{}}
 &:=
 -\frac{1}{\ThinPoiPPIntensity}
 \frac{\partial}{\partial{}X}
 \Bigl(
  \GibbsPPActivity\LinearizedMeasure ( [X,\infty[ )
  +
  \log\PartFunPoisson{\GibbsPPActivity,[X,\infty[,\BC{}\cup{}\RealBis{}}
 \Bigr)
 \\ & =
\label{eq_thinproba_quotient}
 \frac{\lambda}{\ThinPoiPPIntensity}
 e^{-\Hamiltonian{X}(X\mid{}\BC{}\cup{}\RealBis{})}
 \frac{%
  \PartFunPoisson{\GibbsPPActivity,]X,\infty[,\BC{}\cup{}\RealBis{}\cup{}X}
 }{%
  \PartFunPoisson{\GibbsPPActivity,[X,\infty[,\BC{}\cup{}\RealBis{}}
 }
 \,.
\end{align}
The derivative in~\eqref{eq_thinproba_logarithm} is as in~\eqref{eq_derivative_def} and is proven in Section~\ref{sec_derivative}.
The thinning probabilities and the thinning itself are measurable in the boundary condition.
\end{Prop}
\begin{proof}
Consider the points of a realisation $\Real{}$ of $\PoissonPP{\SDom,\ThinPoiPPParams}$ sequentially.
The decision of whether to keep or thin a point $X$ depends only on decisions already taken in $]-\infty,X[$.

In particular, the only information we admit is the location of the already kept points $\RealBis{}\subseteq{}\Real{}\cap]-\infty,X[$.
We name the thinning probability $\PointThinProba{\SDom,\BC{}}{X}{\RealBis{}}$.

\par
We consider what happens if, starting at some $X\in\SDom$, we delete all points in
$\Real{[X,\infty]}$, i.e., all not yet considered points in $\Real{}$.
On the one side, this is the \emph{void probability}, i.e.,  the probability of the empty configuration, of a thinned Poisson PP with intensity $\ThinPoiPPIntensity \PointThinProba{\SDom,\BC{}}{.}{\RealBis{}}$.

\begin{subequations}
\label{eq_cond_void}
\begin{equation}
\label{eq_cond_void_poi_thinned}
 \PoissonPP{[X,\infty[,\ThinPoiPPIntensity \PointThinProba{\SDom,\BC{}}{.}{\RealBis{}},\ThinPoiPPRadiusMeasure}
 (\xi=\emptyset)
 =
 \exp\left(
  - \ThinPoiPPIntensity
  \int_{[X,\infty[}
  \PointThinProba{\SDom,\BC{}}{Y}{\RealBis{}}
  \LinearizedMeasure(\D{}Y)
 \right)
 \,.
\end{equation}
On the other side, the resulting empty realisation follows the
$\GibbsPP{\SDom,\BC{}}$ law.
The DLR equations~\eqref{eq_dlr} imply that
\begin{multline}
\label{eq_cond_void_crcm}
 \GibbsPP{\SDom,\BC{}}(\xi_{[X,\infty[} = \emptyset
  \mid{}\xi_{]-\infty,X[} = \RealBis{})
 \\=
 \GibbsPP{[X,\infty[,\BC{}\cup{}\RealBis{}}(\xi=\emptyset)
 =
 \frac{e^{-\GibbsPPActivity \LinearizedMeasure ( [X,\infty[ ) }}{\PartFunPoisson{\GibbsPPActivity,[X,\infty[,\BC{}\cup{}\RealBis{}}}
 \,.
\end{multline}
\end{subequations}
Equating the left hand sides of~\eqref{eq_cond_void_poi_thinned} and~\eqref{eq_cond_void_crcm} leads to
\begin{equation}
\label{eq_cond_void_equality}
 \ThinPoiPPIntensity
 \int_{[X,\infty[}
  \PointThinProba{\SDom,\BC{}}{Y}{\RealBis{}}
  \LinearizedMeasure(\D{Y})
 =
 \GibbsPPActivity \LinearizedMeasure ( [X,\infty[ )
 +
 \log\PartFunPoisson{\GibbsPPActivity,[X,\infty[,\BC{}\cup{}\RealBis{}}
 \,.
\end{equation}
Taking the derivative along the ordered space $(\SDom,\Order)$ yields~\eqref{eq_thinproba_logarithm}.
On the left-hand side of~\eqref{eq_cond_void_equality} we apply a one-sided version of the \emph{Lebesgue differentiation theorem}~\cite[Thm 5.6.2]{Bogachev__MeasureTheory__Springer_2007} to extract the integrand as the $\LinearizedMeasure$-a.e. derivative.
There is an additional minus sign in~\eqref{eq_thinproba_logarithm}, because differentiation of the left-hand side of~\eqref{eq_cond_void_equality} proceeds in decreasing direction in $\Order$, reverse to the direction used the common direction used in the derivative~\eqref{eq_derivative_def}.
Because $\LinearizedMeasure$ is diffuse, the negligible change of the left interval border from closed to open in the left-hand side of~\eqref{eq_cond_void_crcm} does not matter.

For the moment let us assume the equality \eqref{eq_thinproba_quotient} is true, which is properly proved in Section \ref{sec_derivative}.
It remains to show that the first marginal of $\CouplingThinOne{\SDom,\BC{},\ThinPoiPPParams}$ equals the Gibbs specification $\GibbsPP{\SDom,\BC{}}$.
For $\Real{}\in\Conf{\SDom}$,
\begin{equation*}
 \CouplingThinOne{\SDom,\BC{},\ThinPoiPPParams}(\xi^1=\D{}\Real{})
 =
 \int_{\Conf{\SDom}}
  \JointThinProba{\SDom,\BC{}}{\Real{}}{\RealBis{}}
  \PoissonPP{\SDom,\ThinPoiPPParams}(\D{}\RealBis{})
 \,.
\end{equation*}
Order the points in $\Real{}$ increasingly in $\LTE{}$ and denote them by $Y_1$ to $Y_n$.
Let $\Real{i}:=\EnumSet{Y_1,\dotsc,Y_i}$ with $\Real{0}:=\emptyset$ and $\Real{n+1}:=\Real{}$.
Set $Y_0:=-\infty$ and $Y_{n+1}:=+\infty$.
Expand using~\eqref{eq_thinprobajoint_def} and factorise to get
\begin{multline}
\label{eq_marginal_one}
 \CouplingThinOne{\SDom,\BC{},\ThinPoiPPParams}(\xi^1=\D{}\Real{})
 =
 \left(
 \prod_{i=1}^{n}
 \PointThinProba{\SDom,\BC{}}{Y_i}{\Real{i-1}}
 \right)
 e^{\ThinPoiPPIntensity\LinearizedMeasure(\SDom)}
 \PoissonPP{\SDom,\ThinPoiPPParams}(\D{}\Real{})
 \\
 \times\prod_{i=1}^{n+1}
 \int_{\Conf{]Y_{i-1},Y_i[}}
  \prod_{Z\in{}\RealBis{}}
  \left(1-\PointThinProba{\SDom,\BC{}}{Z}{\Real{i-1}}\right)
  \PoissonPP{]Y_{i-1},Y_i[,\ThinPoiPPParams}(\D{}\RealBis{})
 \,.
\end{multline}
For each factor in the second product, the thinning kernel used within the integral does not change and, using standard Poisson computations and~\eqref{eq_cond_void}, we obtain that
\begin{align*}
 \int_{]Y_{i-1},Y_i[}
 \prod_{Z \in{} \RealBis{} }
 & 
 \left( 1-\PointThinProba{\SDom,\BC{}}{Z}{\Real{i-1}}\right)
 \PoissonPP{]Y_{i-1},Y_i[,\ThinPoiPPParams}(\D{}\RealBis{})
 \\
 &={}
 \exp \left(- \ThinPoiPPIntensity \int_{]Y_{i-1},Y_i[}
 \PointThinProba{\SDom,\BC{}}{Z}{\Real{i-1}}
 \LinearizedMeasure(\D{} Z)
 \right)
 \\
 &={}
 e^{-\GibbsPPActivity \LinearizedMeasure ( ]Y_{i-1},Y_i[ ) }
 \frac{
  \PartFunPoisson{\GibbsPPActivity,]Y_{i},\infty[,\BC{}\cup{}\Real{i-1}}
 }{
  \PartFunPoisson{\GibbsPPActivity,]Y_{i-1},\infty[,\BC{}\cup{}\Real{i-1}}
 }
 \,.
\end{align*}
Furthermore, we can simplify
\begin{align*}
 \prod_{i=1}^{n}
 \frac{%
 \PartFunPoisson{
  \GibbsPPActivity,]Y_i,\infty[,\BC{}\cup{}\Real{i}}
 }{
  \PartFunPoisson{\GibbsPPActivity,]Y_i,\infty[,\BC{}\cup{}\Real{i-1}}
 }
 \prod_{i=1}^{n+1}
 \frac{
  \PartFunPoisson{\GibbsPPActivity,]Y_{i},\infty[,\BC{}\cup{}\Real{i-1}}
 }{
  \PartFunPoisson{\GibbsPPActivity,]Y_{i-1},\infty[,\BC{}\cup{}\Real{i-1}}
 }
 =
 \frac{1}{\PartFunPoisson{\GibbsPPActivity,\SDom,\BC{}}}
 \,.
\end{align*}
Plugging~\eqref{eq_thinproba_quotient} into~\eqref{eq_marginal_one} and adjusting the interval borders to open intervals, then simplifying factors as in the preceding calculations leaves us with
\begin{multline*}
 \CouplingThinOne{\SDom,\BC{},\ThinPoiPPParams}(\xi^1=\D{}\Real{})
 ={}
 e^{(\ThinPoiPPIntensity -\GibbsPPActivity)\LinearizedMeasure(\SDom)}
 \left(\frac{\GibbsPPActivity}{\ThinPoiPPIntensity}\right)^{\Cardinality{\Real{}}}
 \frac{
  e^{-\Hamiltonian{\SDom}(\Real{}|\BC{})}
 }{
  \PartFunPoisson{\GibbsPPActivity,\SDom,\BC{}}
 }
 \PoissonPP{\SDom,\ThinPoiPPParams}(\D{}\Real{})
 \\
 ={}
 \frac{
  e^{-\Hamiltonian{\SDom}(\Real{}|\BC{})}
 }{
  \PartFunPoisson{\GibbsPPActivity,\SDom,\BC{}}
 }
 \PoissonPP{\SDom,\GibbsPPActivity,\GibbsPPRadiusMeasure}(\D{}\Real{})
 =
 \GibbsPP{\SDom,\BC{}}(\D{}\Real{})
 \,.
\end{multline*}
\end{proof}

\subsection{Derivative}
\label{sec_derivative}

This section shows that the thinning probability~\eqref{eq_thinproba_logarithm} is a quotient of partition functions~\eqref{eq_thinproba_quotient}.
Abbreviating $\BC{}':=\BC{}\cup{}\RealBis{}$ and multiplying by $-\alpha$, this reduces to showing that
\begin{multline}
\label{eq_derivative_quotient_equality}
 \frac{\partial}{\partial{}X}
 \Bigl(
  \GibbsPPActivity \LinearizedMeasure( [X,\infty[ )
  +
  \log\PartFunPoisson{\GibbsPPActivity,[X,\infty[,\BC{}'}
 \Bigr)
 \\=
 -\GibbsPPActivity
 e^{-\Hamiltonian{X}(X\mid{}\BC{}')}
 \frac{%
  \PartFunPoisson{\GibbsPPActivity,]X,\infty[,\BC{}'\cup{}X}
 }{%
  \PartFunPoisson{\GibbsPPActivity,[X,\infty[,\BC{}'}
 }
 \,.
\end{multline}
Writing
$z(X):=\PartFunPoisson{\GibbsPPActivity,[X,\infty[,\BC{}'}$, we have the classical relation
\begin{equation*}
 \frac{\partial}{\partial{}X}\log{}z(X)
 = \frac{z'(X)}{z(X)}
 \,.
\end{equation*}
Since
$\frac{\partial}{\partial{}X}\Bigl( \GibbsPPActivity \LinearizedMeasure ( [X,\infty[ ) \Bigr)= - \GibbsPPActivity$,  to show~\eqref{eq_derivative_quotient_equality} it remains to verify that
\begin{equation}
\label{eq_derivative_partfun}
 z'(X)
 =
\GibbsPPActivity z(X)
-\GibbsPPActivity
e^{-\Hamiltonian{X}(X\mid{}\BC{}')}
\PartFunPoisson{\GibbsPPActivity,]X,\infty[,\BC{}'\cup{}X}
 \,.
\end{equation}
We use the notation from~\eqref{eq_derivative_def}.
To lighten the notation and since the Hamiltonian satisfies additivity and is independent of the domain~\eqref{eq_hamiltonian_basic}, we omit its domain subscript for the remainder of this subsection.
Recall that $\varepsilon=\LinearizedMeasure([X,\XPlusEps[)$.
Then,
\begin{align*}
 &z'(X)
 \\\stackrel{}{=}{}
 &\lim_{\varepsilon\to{}0}
  \frac{1}{\varepsilon}
  \left(z(\XPlusEps)-z(X)\right)
 \\\stackrel{}{=}{}
 &\lim_{\varepsilon\to{}0}\frac{1}{\varepsilon}
  \left(
   \int
   e^{-\Hamiltonian{}(\Real{} \mid{} \BC{}')}
   \PoissonPP{[\XPlusEps,\infty[,\GibbsPPParams}(\D{}\Real{})
   -
   \int
   e^{-\Hamiltonian{}(\Real{} \mid{} \BC{}')}
   \PoissonPP{[X,\infty[,\GibbsPPParams}(\D{}\Real{})
  \right)
 \\\stackrel{}{=}{}
 &\lim_{\varepsilon\to{}0}\frac{1}{\varepsilon}
  \int
  \underbrace{
  \int
  \left(
   e^{-\Hamiltonian{}(\Real{} \mid{} \BC{}')}
   -
   e^{-\Hamiltonian{}(\Real{}\cup{}\BC{}'' \mid{} \BC{}')}
  \right)
  \PoissonPP{[\XPlusEps,\infty[,\GibbsPPParams}(\D{}\Real{})
  }_{=:E(\BC{}'')}
  \PoissonPP{[X,\XPlusEps[,\GibbsPPParams}(\D{}\BC{}'')
 \,.
\end{align*}
We split the outermost integral according to the cardinality of $\BC{}''$.
Summing up those cases yields~\eqref{eq_derivative_partfun}, with only the case $\Cardinality{\BC{}''}=1$ having a non-trivial contribution.

\par
$\bullet$
Case $\Cardinality{\BC{}''}=0$:
This is equivalent to $\BC{}''=\emptyset$.
This gives $E(\emptyset)=0$ and
\begin{align*}
 \lim_{\varepsilon\to{}0}\frac{1}{\varepsilon}
 \int
 \Indicator{\BC{}''=\emptyset}
 E(\emptyset)
 \PoissonPP{[X,\XPlusEps[,\GibbsPPParams}(\D{}\BC{}'')
 =
 0
 \,.
\end{align*}

\par
$\bullet$
Case $\Cardinality{\BC{}''}\ge 2$:
For each $\widetilde{\SDom}\in\ProjectionBoundedBorelSets$, the Poisson measures at intensities $\ThinPoiPPIntensity$ and $\GibbsPPActivity$ relate as
\begin{equation}
\label{eq_poi_intensity_change}
 e^{(\ThinPoiPPIntensity-\GibbsPPActivity)\LinearizedMeasure(\widetilde{\SDom})}
 \PoissonPP{\widetilde{\SDom},\ThinPoiPPIntensity,\GibbsPPRadiusMeasure}
 (\D{}\Real{})
 =
 \left( \frac{\ThinPoiPPIntensity}{\GibbsPPActivity} \right)^
 {\Cardinality{\Real{}}}
 \PoissonPP{\widetilde{\SDom},\GibbsPPParams}(\D{}\Real{})
 \,.
\end{equation}
The uniform upper bound on the Papangelou intensity~\eqref{eq_papangelou_bounded_uniform}
and~\eqref{eq_poi_intensity_change} applied to $[\XPlusEps,\infty[$ yields the upper bound
\begin{align*}
 \Modulus{E(\BC{}'')}
 \le
 \int
  \left(\frac{\ThinPoiPPIntensity}{\GibbsPPActivity} \right)^
  {\Cardinality{\Real{}}}
 \left(1 + \left(\frac{\ThinPoiPPIntensity}{\GibbsPPActivity} \right)^
 {\Cardinality{\BC{}''}}\right)
 & \PoissonPP{[\XPlusEps,\infty[,\GibbsPPParams}(\D{}\Real{})
 \\ &=
 \left(1 + \left(\frac{\ThinPoiPPIntensity}{\GibbsPPActivity}\right)^
 {\Cardinality{\BC{}''}}\right)
 e^{(\ThinPoiPPIntensity-\GibbsPPActivity)\LinearizedMeasure([\XPlusEps,\infty[) }
 \,.
\end{align*}
Plugging this inequality back into the outer integral and another application of~\eqref{eq_poi_intensity_change} for $[X,\XPlusEps[$ yields
\begin{align*}
& \LRModulus{
  \int
   \Indicator{\BC{}''\geq 2}
   E(\BC{}'')
   \PoissonPP{[X,\XPlusEps[,\GibbsPPParams}(\D{}\BC{}'')
 }
 \\ & \le{}
 \int \Indicator{\BC{}''\geq 2}
  \left(1 + \left(\frac{\ThinPoiPPIntensity}{\GibbsPPActivity}\right)^
  {\Cardinality{\BC{}''}}\right)
  e^{(\ThinPoiPPIntensity-\GibbsPPActivity)\LinearizedMeasure([\XPlusEps,\infty[) }
  \PoissonPP{[X,\XPlusEps[,\GibbsPPParams}(\D{}\BC{}'')
 \\ & =
 e^{(\ThinPoiPPIntensity-\GibbsPPActivity)\LinearizedMeasure([\XPlusEps,\infty[) }
 \left(
 \PoissonPP{[X,\XPlusEps[,\GibbsPPParams}(\Cardinality{\xi}\ge 2)
 +
 e^{(\ThinPoiPPIntensity-\GibbsPPActivity)\varepsilon}
 \PoissonPP{[X,\XPlusEps[,\ThinPoiPPIntensity,\GibbsPPRadiusMeasure}(\Cardinality{\xi}\ge 2)
 \right)
 \,.
\end{align*}
As both $\PoissonPP{[X,\XPlusEps[,\ThinPoiPPIntensity,\GibbsPPRadiusMeasure}(\Cardinality{\xi}\ge 2)$ and $\PoissonPP{[X,\XPlusEps[,\GibbsPPParams}(\Cardinality{\xi}\ge 2)$ are $o(\varepsilon^2)$, we see that
\begin{equation*}
 \lim_{\varepsilon\to{}0}\frac{1}{\varepsilon}
 \LRModulus{
  \int
  \Indicator{\Cardinality{\BC{}''}\ge{}2}
  E(\BC{}'')
  \PoissonPP{[X,\XPlusEps[,\GibbsPPParams}(\D{}\BC{}'')
 }
 = 0
 \,.
\end{equation*}

\par
$\bullet$
Case $\Cardinality{\BC{}''} = 1$:
Let $Y$ be the single point in the configuration $\BC{}''$.
Then
\begin{align*}
 &
 \frac{1}{\varepsilon}
 \int
 \Indicator{\BC{}'' = 1}
 E(\BC{}'')
 \PoissonPP{[X,\XPlusEps[,\GibbsPPParams}(\D{}\BC{}'')
 \\\stackrel{}{=}{}
 &
 \frac{\GibbsPPActivity e^{-\GibbsPPActivity \varepsilon }}{\varepsilon}
 \int_{[X,\XPlusEps[}
 \int
  \left(
   e^{-\Hamiltonian{}(\Real{} \mid{} \BC{}')}
   -
   e^{-\Hamiltonian{}(\Real{}\cup{} Y \mid{} \BC{}')}
  \right)
 \PoissonPP{[\XPlusEps,\infty[,\GibbsPPParams}(\D{}\Real{})
 \LinearizedMeasure(\D{}Y)
 \\\stackrel{}{=}{}
 &
 \frac{\GibbsPPActivity e^{-\GibbsPPActivity \varepsilon }}{\varepsilon}
  \int_{[X,\XPlusEps[}
   \left(
   z(\XPlusEps)
   -
   \int
    e^{-\Hamiltonian{}(\Real{}\cup{} Y \mid{} \BC{}')}
    \PoissonPP{[\XPlusEps,\infty[,\GibbsPPParams}(\D{}\Real{})
      \right)
   \LinearizedMeasure(\D{}Y)
 \,.
\end{align*}
First, we analyse the integral over the left integrand.
Using the continuity of $z$ we obtain
\begin{align*}
\frac{\GibbsPPActivity e^{-\GibbsPPActivity \varepsilon }}{\varepsilon}
\int_{[X,\XPlusEps[}
z(\XPlusEps)
\LinearizedMeasure(\D{}Y)
=
\GibbsPPActivity e^{-\GibbsPPActivity \varepsilon }
z(\XPlusEps)
\underset{\varepsilon \to 0}{\longrightarrow}
\GibbsPPActivity z(X).
\end{align*}

Second, we analyse the integral over the right integrand.
Using~\eqref{eq_hamiltonian_basic} to expand
$
 \Hamiltonian{}(\Real{}\cup{} Y \mid{} \BC{}')
 =
 \Hamiltonian{}(\Real{} \mid{} \BC{}')
 +
 \Hamiltonian{}( Y \mid{} \Real{} \cup{} \BC{}')
$,
we get
\begin{align*}
 &
 \frac{\GibbsPPActivity e^{-\GibbsPPActivity \varepsilon }}{\varepsilon}
 \int_{[X,\XPlusEps[}
  \int
   e^{-\Hamiltonian{}(\Real{}\cup{} Y \mid{} \BC{}')}
   \PoissonPP{[\XPlusEps,\infty[,\GibbsPPParams}(\D{}\Real{})
  \LinearizedMeasure(\D{}Y)
 \\\stackrel{}{=}{}
 &
 \frac{\GibbsPPActivity e^{-\GibbsPPActivity \varepsilon }}{\varepsilon}
 \int
  e^{ - \Hamiltonian{}(\Real{} \mid{} \BC{}')}
  \int_{[X,\XPlusEps[}
   e^{- \Hamiltonian{}( Y \mid{} \Real{} \cup{} \BC{}') }
   \LinearizedMeasure(\D{}Y)
  \PoissonPP{[\XPlusEps,\infty[,\GibbsPPParams}(\D{}\Real{})
 \\\stackrel{}{=}{}
 &
 \frac{\GibbsPPActivity }{\varepsilon}
 \int
  \Indicator{\Real{[X,\XPlusEps[}=\emptyset}
  e^{ - \Hamiltonian{}(\Real{} \mid{} \BC{}')}
  \int_{[X,\XPlusEps[}
   e^{- \Hamiltonian{}( Y \mid{} \Real{} \cup{} \BC{}') }
   \LinearizedMeasure(\D{}Y)
  \PoissonPP{]X,\infty[,\GibbsPPParams}(\D{}\Real{})
 \,.
\end{align*}
The functions
\begin{equation*}
 \varepsilon \mapsto
 \Indicator{\Real{[X,\XPlusEps[}=\emptyset}
 e^{ - \Hamiltonian{}(\Real{} \mid{} \BC{}')}
\end{equation*}
and
\begin{equation*}
 \varepsilon \mapsto
 \frac{1}{\varepsilon}
 \int_{[X,\XPlusEps[}
  e^{- \Hamiltonian{}( Y \mid{} \Real{} \cup{} \BC{}') }
    \LinearizedMeasure(\D{}Y)
\end{equation*}
are $\PoissonPP{[\XPlusEps,\infty[,1,\ThinPoiPPRadiusMeasure}$-a.s. continuous in $\varepsilon=0$ and dominated, thanks to~\eqref{eq_papangelou_bounded_uniform} by  integrable functions.
Using a standard continuity theorem and a one-sided version of the \emph{Lebesgue differentiation theorem}~\cite[Thm 5.6.2]{Bogachev__MeasureTheory__Springer_2007}, we obtain
\begin{align*}
 &
 \lim_{\varepsilon\to{}0}
 \frac{\GibbsPPActivity e^{-\GibbsPPActivity \varepsilon }}{\varepsilon}
 \int_{[X,\XPlusEps[}
 \int
  e^{-\Hamiltonian{}(\Real{}\cup{} Y \mid{} \BC{}')}
  \PoissonPP{[\XPlusEps,\infty[,\GibbsPPParams}(\D{}\Real{})
  \GibbsPPActivity{}e^{-\GibbsPPActivity\varepsilon}
  \LinearizedMeasure(\D{}Y)
 \\\stackrel{}{=}{}
 &
 \int
  \lim_{\varepsilon\to{}0}
   \frac{\GibbsPPActivity }{\varepsilon}
   \Indicator{\Real{[X,\XPlusEps[}=\emptyset}
   e^{-\Hamiltonian{}(\Real{} \mid{} \BC{}')}
   \int_{[X,\XPlusEps[}
    e^{-\Hamiltonian{}( Y \mid{} \Real{} \cup{} \BC{}') }
    \LinearizedMeasure(\D{}Y)
 \PoissonPP{]X,\infty[,\GibbsPPParams}(\D{}\Real{})
 \\\stackrel{}{=}{}
 &
 \GibbsPPActivity
 \int
  e^{-\Hamiltonian{}(\Real{} \mid{} \BC{}')}
  e^{-\Hamiltonian{}(X \mid{} \Real{} \cup{} \BC{}') }
 \PoissonPP{]X,\infty[,\GibbsPPParams}(\D{}\Real{})
 \\\stackrel{}{=}{}
 &
 \GibbsPPActivity
 e^{- \Hamiltonian{}( X \mid{} \BC{}') }
 \int
  e^{ - \Hamiltonian{}(\Real{} \mid{} \BC{}' \cup{} X)}
  \PoissonPP{]X,\infty[,\GibbsPPParams}(\D{}\Real{})
 \\\stackrel{}{=}{}
 &
 \GibbsPPActivity
 e^{-\Hamiltonian{}(X\mid{}\BC{}')}
 \PartFunPoisson{\GibbsPPActivity,]X,\infty[,\BC{}'\cup{}X}
 \,.
\end{align*}
Adding the right and left terms of this case together with the zeros from the other two cases  gives~\eqref{eq_derivative_partfun}, which shows~\eqref{eq_derivative_quotient_equality}.

\subsection{Construction of the disagreement coupling}
\label{sec_dac_construction}

Let $\SDom\in\ProjectionBoundedBorelSets$ and $\BC{}^1,\BC{}^2\in\Conf{\Complement{\SDom}}$.
We thin a Poisson PP to two conditionally independent copies of the Gibbs PP.
This coupling
\footnote{The superscript ``thin2'' stands for ``joint thinning to two Gibbs PPs''.}
on $(\Conf{\SDom}^3,\ProductAlgebra{\SDom}{3})$ is given by
\begin{equation}
\label{eq_thintwo_janossy}
 \CouplingThinTwo{\SDom,\BC{}^1,\BC{}^2}
 (\D{}(\Real{}^1,\Real{}^2,\Real{}^3))
 :=
 \JointThinProba{\SDom,\BC{}^1}{\Real{}^1}{\Real{}^3}
 \JointThinProba{\SDom,\BC{}^2}{\Real{}^2}{\Real{}^3}
 \PoissonPP{\SDom,\ThinPoiPPParams}(\D{}\Real{}^3)
 \,.
\end{equation}
The coupling $\CouplingThinTwo{\SDom,\BC{}^1,\BC{}^2}$ fulfils everything in~\eqref{eq_dacf_properties} except the crucial~\eqref{eq_dacf_connected}.
Hence, our approach is to use $\CouplingThinTwo{\SDom,\BC{}^1,\BC{}^2}$ only where and when the boundary conditions have a direct influence and recurse until the boundary conditions have no influence.
In the no-influence case we thin to one Gibbs PP and identify it with the two target Gibbs PPs.
This recursive approach keeps the disagreement allows connected to some influence from the boundary conditions from the previous steps and ensures that~\eqref{eq_dacf_connected} holds.

\par
The following lines codify the influence of the boundary conditions in~\eqref{eq_zone_def},
describe the restricted use of $\CouplingThinTwo{\SDom,\BC{}^1,\BC{}^2}$ in~\eqref{eq_dac_def_zone} and give the recursive construction of a disagreement coupling in~\eqref{eq_dac_def_rec}.
The \emph{influence zone} is
\begin{equation}
\label{eq_zone_def}
 \InfluenceZone :=
 \DescSet{X\in{}\SDom}
  {\exists{}Y\in{}\BC{}^1\cup{}\BC{}^2:
   \Ball{X}\cap\Ball{Y}\not=\emptyset
  }
 \,.
\end{equation}
Define the joint Janossy intensity of the law
\footnote{The superscript ``da-zone'' stands for ``disagreement, influence zone case''.}
$\CouplingDisAgZone{\SDom,\BC{}^1,\BC{}^2}$
on $(\Conf{\InfluenceZone}^3,\ProductAlgebra{\InfluenceZone}{3})$ by
\begin{subequations}
\label{eq_dac_def}
\begin{equation}
\label{eq_dac_def_zone}
\begin{aligned}
 \CouplingDisAgZone{\SDom,\BC{}^1,\BC{}^2}
  (\D{}(\Real{}^1,\Real{}^2,\Real{}^3))
 :=
 \CouplingThinTwo{\SDom,\BC{}^1,\BC{}^2}
 (\D{}(\Real{\InfluenceZone}^1,\Real{\InfluenceZone}^2,\Real{\InfluenceZone}^3))
 \,.
\end{aligned}
\end{equation}
Recall that $\Indicator{.}$ is the indicator function.
Define the joint Janossy intensity of the law
\footnote{The superscript ``da-rec'' stands for ``disagreement, recursive case''.}
$\CouplingDisAgRec{\SDom,\BC{}^1,\BC{}^2}$ on $(\Conf{\SDom}^3,\ProductAlgebra{\SDom}{3})$ recursively by
\begin{multline}
\label{eq_dac_def_rec}
\CouplingDisAgRec{\SDom,\BC{}^1,\BC{}^2}
 (\D{}(\Real{}^1,\Real{}^2,\Real{}^3))
 \\  :=
 \Indicator{\InfluenceZone=\emptyset}
 \Indicator{\Real{}^1=\Real{}^2}
 \CouplingThinOne{\SDom,\emptyset}(\D{}(\Real{}^1,\Real{}^3))
 \\  +
 \Indicator{\InfluenceZone\not=\emptyset}
 \CouplingDisAgZone{\SDom,\BC{}^1,\BC{}^2}
 (\D{}(\Real{\InfluenceZone}^1,\Real{\InfluenceZone}^2,\Real{\InfluenceZone}^3))
 \\
 \times
 \CouplingDisAgRec
  {\SDom\setminus\InfluenceZone
  ,\BC{}^1\cup\Real{\InfluenceZone}^1
  ,\BC{}^2\cup\Real{\InfluenceZone}^2
  }(\D{}(\Real{\SDom\setminus\InfluenceZone}^1,
  \Real{\SDom\setminus\InfluenceZone}^2,\Real{\SDom\setminus\InfluenceZone}^3))
 \,.
\end{multline}
\end{subequations}

\begin{Prop}
\label{prop_dac_properties}
The coupling $\CouplingDisAgRec{\SDom,\BC{}^1,\BC{}^2}$ fulfils~\eqref{eq_dacf_properties}.
Further, it is jointly measurable in the boundary conditions $(\BC{}^1,\BC{}^2)$.
\end{Prop}

\begin{proof}
The first step is to check the termination of the recursion in~\eqref{eq_dac_def_rec}.
The recursion is made with respect to the influence zone $\InfluenceZone$, which is decreasing and whose volume is bounded by $\LinearizedMeasure(\SDom)$.
The recursion stops when no Gibbs point (of $\xi^1$ and $\xi^2$) is placed in $\InfluenceZone$.
This happens in particular when there is no Poisson point (of $\xi^3$) in the influence zone $\InfluenceZone$.
At each step of the recursion, this happens independently with probability bounded from below by $e^{-\ThinPoiPPIntensity\LinearizedMeasure(\SDom)}$.
Therefore, the recursion stops after an almost-surely finite number of steps.
\par
The next step is to show the measurability in the boundary conditions.
Proposition~\ref{prop_thinning} asserts that $ \JointThinProba{\SDom,\BC{}^1}{\Real{}^1}{\Real{}^3}$ and $\JointThinProba{\SDom,\BC{}^2}{\Real{}^2}{\Real{}^3}$ are measurable in $\BC{}^1$ and $\BC{}^2$ respectively.
Hence, the coupling $\CouplingDisAgRec{\SDom,\BC{}^1,\BC{}^2}$ is jointly measurable in the boundary conditions $(\BC{}^1,\BC{}^2)$.
The measurability is needed for the well-definedness of the recursive definition~\eqref{eq_dac_def_rec} and the proof of~\eqref{eq_dacf_properties}.
\par
Finally, we show that~\eqref{eq_dacf_properties} holds.
Equation~\eqref{eq_dacf_gibbs} is a straightforward consequence of the DLR equations~\eqref{eq_dlr} and the assumption~\eqref{eq_hamiltonian_locality}.
Properties~\eqref{eq_dacf_poi} and~\eqref{eq_dacf_domination} are also a straightforward consequence of the construction.
Concerning~\eqref{eq_dacf_connected}, the only points of the Poisson configuration $\xi^3$ which are not connected (in the Gilbert graph $\GilbertGraph{\xi^3}$) to the boundary conditions $\BC{}^1 \cup \BC{}^2$ are the ones sampled at the end of the recursion, when $\InfluenceZone = \emptyset$.
These points thin to both Gibbs PPs $\xi^1$ and $\xi^2$ identically, as outlined in the $\InfluenceZone=\emptyset$ case of~\eqref{eq_dac_def_rec}.
By construction~\eqref{eq_dac_def_zone}, from those points the ones belonging to the first Gibbs configuration $\xi^1$ also belong to the second Gibbs configuration $\xi^2$.
Therefore, the only points where the two Gibbs configurations may differ are the ones sampled when the influence zone is not empty.
By~\eqref{eq_dac_def_zone}, such points are connected to the boundary conditions $\BC{}^1 \cup \BC{}^2$.
\end{proof}

\section{Disagreement percolation proofs}
\label{sec_dap_proofs}

\par
This section contains the proofs of  Theorem~\ref{thm_unique_gibbs} and Theorem~\ref{thm_correlations}.

\par
Let $\Domein_1,\Domein_2\in\BoundedBorelSets$ with $\Domein_1\subsetneq\Domein_2$ and $\BC{}^1,\BC{}^2 \in \Conf{\Domein_2^c}$, as well as an event $E\in\SigmaAlgebra{\Domein_1}$.
Property~\eqref{eq_dacf_gibbs} introduces the disagreement coupling to express the difference between the probabilities of two Gibbs states for $E$.
\begin{multline}
\label{eq_disagreement_intro}
 \Modulus{
  \GibbsPP{\Domein_2,\BC{}^1}(E)
  -
  \GibbsPP{\Domein_2,\BC{}^2}(E)
 }
 \\
 \stackrel{}{\leq}{}
 \max\LREnumSet{
  \Dacf{\Domein_2,\BC{}^1,\BC{}^2}
  (\xi_{\Domein_1}^1 \in E, \xi_{\Domein_1}^2 \not \in E)
  ,
  \Dacf{\Domein_2,\BC{}^1,\BC{}^2}
  (\xi_{\Domein_1}^2 \in E, \xi_{\Domein_1}^1 \not \in E)
 }
 \,.
\end{multline}
The other properties of a disagreement coupling~\eqref{eq_dacf_properties} allow to bound each of the above terms by the same percolation probability.
The connection event is increasing and its probability increases under the dominating Poisson PP.
We only show the first case.
\begin{equation}
\label{eq_disagreement_classic_steps}
\begin{aligned}
 \Dacf{\Domein_2,\BC{}^1,\BC{}^2}
  (\xi_{\Domein_1}^1 \in E, \xi_{\Domein_1}^2 \not \in E)
 \stackrel{}{\leq}{}
 &
 \Dacf{\Domein_2,\BC{}^1,\BC{}^2}
 (\xi_{\Domein_1}^1\SymDiff{}\xi_{\Domein_1}^2
  \not = \emptyset )
 \\
 \stackrel{}{\leq}{}
 &
 \Dacf{\Domein_2,\BC{}^1,\BC{}^2}
 (\Domein_1
  \ConnectedIn{\xi^3}
  \BC{}^1\cup{}\BC{}^2
 )
 \\
 \stackrel{}{=}{}
 &
 \PoissonPP{\Domein_2,\DacPoiPPParams}
 (\Domein_1
  \ConnectedIn{\xi}
  \BC{}^1\cup{}\BC{}^2
 )
 \,.
\end{aligned}
\end{equation}
Hence,
\begin{equation}
\label{eq_disagreement_classic_bound}
 \Modulus{
  \GibbsPP{\Domein_2,\BC{}^1}(E)
  -
  \GibbsPP{\Domein_2,\BC{}^2}(E)
 }
 \le
 \PoissonPP{\Domein_2,\DacPoiPPParams}
  (\Domein_1\ConnectedIn{\xi}\BC{}^1\cup{}\BC{}^2)
 \,.
\end{equation}

\subsection{Proof of Theorem~\ref{thm_unique_gibbs}}
\label{sec_thm_unique_gibbs_proof}

Let $\GenericPP{1}{},\GenericPP{2}{} \in \GibbsStates$.
We want to prove that $\GenericPP{1}{} = \GenericPP{2}{}$.
Let $\emptyset\not=\Domein\in\BoundedBorelSets$ and $E\in\SigmaAlgebra{\Domein}$.
For $n\in\NN$, consider the closed ball $\Domein_n := \Ball{0,n}$ in $\RRd$.
Let $\GenericPP{1\TensorTimes{}2}{n}:=\GenericPP{1}{\Complement{\Domein_n}}\TensorTimes\GenericPP{2}{\Complement{\Domein_n}}$.
For $n$ large enough, $\Domein\subseteq\Domein_n$.
The DLR equation~\eqref{eq_dlr} for $\Domein_n$ implies that
\begin{equation*}
 \Modulus{\GenericPP{1}{}(E) - \GenericPP{2}{}(E)}
 \leq
 \int
 \Modulus{
  \GibbsPP{\Domein_n,\BC{}^1}(E)
  -
  \GibbsPP{\Domein_n,\BC{}^2}(E)
 }
 \GenericPP{1\TensorTimes{}2}{n}(\D{}(\BC{}^1,\BC{}^2))
 \,.
\end{equation*}
Applying~\eqref{eq_disagreement_classic_bound} yields
\begin{equation*}
 \Modulus{\GenericPP{1}{}(E) - \GenericPP{2}{}(E)}
 \leq{}
 \int
  \PoissonPP{\Domein_n,\DacPoiPPParams}
   (\Domein\ConnectedIn{\xi}\BC{}^1\cup\BC{}^2 )
  \GenericPP{1\TensorTimes{}2}{n}(\D{}(\BC{}^1,\BC{}^2))
 \,.
\end{equation*}
As we are in the sub-critical regime of the Boolean model, we expect the integrated probability to converge to $0$ as $n$ grows to infinity.
Unfortunately, this convergence depends on the outside configurations $\BC{}^1, \BC{}^2 $ and we need uniform convergence.

Let $\varepsilon >0$.
Since the integrated event is increasing in $\BC{}^1$ and $\BC{}^2$, the stochastic domination of both $\GenericPP{1}{}$ and $\GenericPP{2}{}$ by $\PoissonPP{\DacPoiPPParams}$ implies that
\begin{multline*}
 \Modulus{\GenericPP{1}{}(E) - \GenericPP{2}{}(E)}
 \leq
 \int
 \PoissonPP{\Domein_n,\DacPoiPPParams}
  (\Domein\ConnectedIn{\xi}\BC{}^1\cup\BC{}^2)
 \PoissonPP{\Complement{\Domein_n},\DacPoiPPParams}
 \TensorTimes
 \PoissonPP{\Complement{\Domein_n},\DacPoiPPParams}
 (\D{}(\BC{}^1,\BC{}^2))
 \,.
\end{multline*}

The following lemma shows how to control with high probability the radii in a Boolean model.
\begin{Lem}
\label{lem_disconnect}
For a positive integer $k$, let
\begin{align*}
 \Upsilon_k :=
 \LRDescSet
  {\Real{} \in \Conf{}}
  {\forall (x,r) \in \Real{},r \leq \frac{\DistanceBetween{x}{0}}{2} + k}
  \,.
\end{align*}
If $\RadiusMeasure$ satisfies the integrability assumption $\IntCondRad{\RadiusMeasure}< \infty$ and $k$ is large enough, then
\begin{align}
 \PoissonPP{\alpha,\GibbsPPRadiusMeasure}(\Upsilon_k)
 \geq 1 - \varepsilon
 \,.
\end{align}

\end{Lem}
This proof is an adaptation of~\cite[Lemma 3.3]{Dereudre__TheExistenceOfQuermassInteractionProcessesForNonlocallyStableInteractionAndNonboundedConvexGrains__AAP_2009}.
If $X\in\BC{} ^1 \cup \BC{} ^2 \in \Upsilon_k$, then
$\Ball{X} \cap \Ball{0,n/2 -k-1} = \emptyset$.
Therefore, for large enough $k$, we have
\begin{align*}
 \Modulus{\GenericPP{1}{} (E) - \GenericPP{2}{}(E)}
 &\leq
\varepsilon +
 \PoissonPP{\Domein_n,\DacPoiPPParams}
  (\Domein\ConnectedIn{\xi }\Ball{0,n/2-k-1})
 \,.
\end{align*}
Using~\eqref{eq_boolean_subcrit_connect_vanish} from Theorem~\ref{thm_percolation_boolean}, for large enough $n$, we have
\begin{align*}
 \Modulus{\GenericPP{1}{} (E) - \GenericPP{2}{}(E)}
 \leq 2 \varepsilon
 \,.
\end{align*}
Letting $\varepsilon$ tend to $0$ shows that $\GenericPP{1}{}=\GenericPP{2}{}$.

\subsection{Proof of Theorem~\ref{thm_correlations}}
\label{sec_thm_correlations_proof}

\par
Recall that $\DacPoiPPRadiusMeasure$ has bounded support, i.e.,
$\DacPoiPPRadiusMeasure([0,r_0])=1$, for some $r_0\in\RRplus$.
As introduced in Section~\ref{sec_space}, $\EuclideanDistance$ is the Euclidean distance between points and/or sets.

\par
First, we prove~\eqref{eq_correlations_spec}.
The DLR equations~\eqref{eq_dlr} let us localise in
\begin{multline*}
 \LRModulus{
  \GibbsPP{\Domein_2, \BC{}} (\xi_{\Domein_1}\in{}E)
  -
  \GenericPP{}{}(\xi_{\Domein_1}\in{}E)
 }
 \\=
 \LRModulus{
   \GibbsPP{\Domein_2, \BC{}} (\xi_{\Domein_1}\in{}E)
   -
   \int
    \GibbsPP{\Domein_2,\BC{}'} (\xi_{\Domein_1}\in{}E)
    \GenericPP{}{}(\xi _{\Domein_2^c}=\D{} \BC{}')
  }
 \\\le
 \int
  \LRModulus{
   \GibbsPP{\Domein_2,\BC{}} (\xi_{\Domein_1}\in{}E)
   -
   \GibbsPP{\Domein_2,\BC{}'}(\xi_{\Domein_1}\in{}E)
  }
  \GenericPP{}{}(\xi _{\Domein_2^c}=\D{} \BC{}')
 \,.
\end{multline*}

For $\Domein \subseteq \RR^\Dimension$, let $\Domein^{\ominus}:= \DescSet{x \in \RR^\Dimension}{\DistanceBetween{x}{\Domein^c}\geq r_0}$ be the reduced set of points not closer than $r_0$ to  the boundary of $\Domein$.
Using~\eqref{eq_disagreement_classic_bound}, we get
\begin{equation*}
\label{eq_proof_correlation_first_2}
\begin{aligned}
 \LRModulus{
  \GibbsPP{\Domein_2, \BC{}} (\xi_{\Domein_1}\in{}E)
  -
  \GenericPP{}{}(\xi_{\Domein_1}\in{}E)
 }
 \stackrel{}{\leq}{}
 &\int
   \PoissonPP{\Domein_2,\DacPoiPPParams}
    (\Domein_1\ConnectedIn{\xi}\BC{}\cup\BC{}')
   \GenericPP{}{}(\xi _{\Domein_2^c}=\D{} \BC{}')
 \\\stackrel{}{\leq}{}
 &\PoissonPP{\Domein_2,\DacPoiPPParams}
   (\Domein_1\ConnectedIn{\xi}(\Domein_2^{\ominus})^c)
 \,.
\end{aligned}
\end{equation*}
Applying~\eqref{eq_boolean_subcrit_exp_decay} from Theorem~\ref{thm_percolation_boolean} results in
\begin{equation*}
 \LRModulus{
  \GibbsPP{\Domein_2, \BC{}} (\xi_{\Domein_1}\in{}E)
  -
  \GenericPP{}{}(\xi_{\Domein_1}\in{}E)
 }
 \leq
 K \exp(- \kappa [\DistanceBetween{\Domein_1}{\Domein_2^c} -  r_0] )
 \,.
\end{equation*}
Setting $K':=Ke^{\kappa{}r_0}$, we obtain~\eqref{eq_correlations_spec}.

\par
Second, we show~\eqref{eq_correlation_event}.
Let $\Domein\in\BoundedBorelSets$ contain $\Domein_1 \cup \Domein_2$ such that
$\DistanceBetween{\Domein_1 \cup \Domein_2}{\Domein^c} \geq
\DistanceBetween{\Domein_1}{\Domein_2}$ and let
$\Domein_3 := \Domein\setminus\Domein_2 $.
Thus, we have $\DistanceBetween{\Domein_1}{\Domein_3^c}
\geq \DistanceBetween{\Domein_1}{\Domein_2}$.
The DLR equations~\eqref{eq_dlr} let us localise to $\Domein_1$ and apply~\eqref{eq_correlation_event} in
\begin{equation}
\label{eq_proof_correlation_second}
\begin{aligned}
 &\BigModulus{
   \GenericPP{}{}(\xi_{\Domein_1\cup\Domein_2}\in{}E_1\cap{}E_2)
   -
   \GenericPP{}{}(\xi_{\Domein_1}\in{}E_1)
   \GenericPP{}{}(\xi_{\Domein_2}\in{}E_2)
  }
 \\\stackrel{}{=}{}
 &\LRModulus{
   \int
   \Indicator{\BC{\Domein_2}\in{}E_2}
   \left(
    \GibbsPP{\Domein_3,\BC{}}(\xi_{\Domein_1}\in{}E_1)
    -
    \GenericPP{}{}(\xi_{\Domein_1}\in{}E_1)
   \right)
   \GenericPP{}{\Domein_3^c} (\D{} \BC{})
  }
 \\\stackrel{}{\leq}{}
 &
  \int
  \LRModulus{
   \GibbsPP{\Domein_3, \BC{}} (\xi_{\Domein_1}\in{}E)
   -
   \GenericPP{}{}(\xi_{\Domein_1}\in{}E)
  }
  \GenericPP{}{\Domein_3^c} (\D{}\BC{})
 \\
 \stackrel{}{\leq}{}
 & K'\exp(-\kappa\,\DistanceBetween{\Domein_1}{\Domein_2})
 \,.
\end{aligned}
\end{equation}

Finally, we prove~\eqref{eq_correlation_pair}.
For this, we need to improve~\eqref{eq_disagreement_classic_bound} for an increasing event $E\in\SigmaAlgebra{\Domein_1}$.
Starting from~\eqref{eq_disagreement_intro}, we redo~\eqref{eq_disagreement_classic_steps}.
We keep $E$ and use the fact that both $E$ and the connection event, and so its intersection, are increasing events.
Of course the properties~\eqref{eq_dacf_properties} are used, too.
\begin{align*}
 \Dacf{\Domein_2,\BC{}^1,\BC{}^2}
  (\xi_{\Domein_1}^1 \in E, \xi_{\Domein_1}^2 \not \in E)
 \stackrel{}{\leq}{}
  &
 \Dacf{\Domein_2,\BC{}^1,\BC{}^2}
 (\xi_{\Domein_1}^3 \in E,
  \Domein_1
  \ConnectedIn{\xi^3
  }  \BC{}^1 \cup \BC{}^2  )
   \\  \stackrel{}{\leq}{}  &
 \Dacf{\Domein_2,\BC{}^1,\BC{}^2}
 (\xi_{\Domein_1}^3 \in E,
  \Domein_1^\oplus
  \ConnectedIn{\xi^3_{\Domein_2\setminus\Domein_1}
  }  (\Domein_2^{\ominus})^c  )
 \\  \stackrel{}{=}{}  &
 \PoissonPP{\Domein_1,\DacPoiPPParams}(E)
 \PoissonPP{\Domein_2\setminus\Domein_1,\DacPoiPPParams}
 (\Domein_1^\oplus
  \ConnectedIn{\xi^3_{\Domein_2\setminus\Domein_1}
  }  (\Domein_2^{\ominus})^c )
   \\  \stackrel{}{\leq}{}  &
 \PoissonPP{\Domein_1,\DacPoiPPParams}(E)
 \PoissonPP{\Domein_2,\DacPoiPPParams}
 (\Domein_1^\oplus
  \ConnectedIn{\xi^3}  (\Domein_2^{\ominus})^c )
 \,,
\end{align*}
where
$\Domein_1^{\oplus}:= \DescSet{x \in \RR^\Dimension}{\DistanceBetween{x}{\Domein_1} \leq r_0}$
is the augmented set containing points not further than $r_0$ from $\Domein_1$.
By applying~\eqref{eq_boolean_subcrit_exp_decay} and increasing the value of $K'$, we arrive at
\begin{equation}
\label{eq_disagreement_increasing_exp}
 \Modulus{
  \GibbsPP{\Domein_1,\BC{}^1}(E)
  -
  \GibbsPP{\Domein_1,\BC{}^2}(E)
 }
 \leq
 \PoissonPP{\Domein_1,\DacPoiPPParams}(E)
 K' \exp(-\kappa\,\DistanceBetween{\Domein_1}{\Domein_2^c })
 \,.
\end{equation}
Retracing~\eqref{eq_proof_correlation_second} with~\eqref{eq_disagreement_increasing_exp} instead of~\eqref{eq_disagreement_classic_bound} yields, for disjoint $\Domein_1,\Domein_2 \in\BoundedBorelSets$,
\begin{multline*}
 \BigModulus{
  \GenericPP{}{}(
   \Cardinality{\xi_{\Domein_1}}\geq n_1,
   \Cardinality{\xi_{\Domein_2}}\geq n_2
   )
  -
  \GenericPP{}{}(\Cardinality{\xi_{\Domein_1}}\geq n_1)
  \GenericPP{}{}(\Cardinality{\xi_{\Domein_2}}\geq n_2)
 }
 \\
 \leq
 K' \exp(-\kappa\,\DistanceBetween{\Domein_1}{\Domein_2})
 \PoissonPP{\Domein_1,\DacPoiPPParams}(\Cardinality{\xi}\geq n_1)
 \PoissonPP{\Domein_2,\DacPoiPPParams}(\Cardinality{\xi}\geq n_2)
 \,.
\end{multline*}
Writing $\Expect{\GenericPP{}{}}$ for the expectation under $\GenericPP{}{}$, the difference between the moments is bounded as
\begin{align*}
 &
 \BigModulus{
  \Expect{\GenericPP{}{}}
   \Cardinality{\xi_{\Domein_1}}
   \Cardinality{\xi_{\Domein_2}}
  -
  \Expect{\GenericPP{}{}}\Cardinality{\xi_{\Domein_1}}
  \Expect{\GenericPP{}{}}\Cardinality{\xi_{\Domein_2}}
 }
 \\\stackrel{}{\le}{}
 &
 \sum_{n_1,n_2 \geq 1}
 \BigModulus{
  \GenericPP{}{}(
   \Cardinality{\xi_{\Domein_1}}\geq n_1,
   \Cardinality{\xi_{\Domein_2}}\geq n_2
   )
   -
  \GenericPP{}{}(\Cardinality{\xi_{\Domein_1}}\geq n_1)
  \GenericPP{}{}(\Cardinality{\xi_{\Domein_2}}\geq n_2)
 }
 \\\stackrel{}{\le}{}
 &
 \sum_{n_1,n_2 \geq 1}
  K' \exp(-\kappa\,\DistanceBetween{\Domein_1}{\Domein_2})
  \PoissonPP{\Domein_1,\DacPoiPPParams}(\Cardinality{\xi}\geq n_1)
  \PoissonPP{\Domein_2,\DacPoiPPParams}(\Cardinality{\xi}\geq n_2)
 \\\stackrel{}{=}{}
 &
 \DacPoiPPIntensity^2 \Leb(\Domein_1)\Leb(\Domein_2)
 K' \exp(-\kappa\,\DistanceBetween{\Domein_1}{\Domein_2})
 \,.
\end{align*}
The result follows by disintegrating with respect to $\Leb^2$.

\section*{Acknowledgements}

This work was supported in part by the GDR 3477 ``Géométrie stochastique'', the ANR "Percolation et percolation de premier passage" (ANR-16-CE40-0016) and the Labex CEMPI (ANR-11-LABX-0007-01).

Both authors thank the anonymous referees for their work.
The first author thanks Marie-Colette van Lieshout for being always helpful with the general theory of point processes and the GDR 3477 for organising the workshop on continuum percolation in 2016 and the SGSIA 2017, both of which made this work possible.
The second author thanks the CWI for its hospitality.

We thank G. Last for pointing out an error in the published version fixed in the arXiv version.

\bibliographystyle{plain}
\bibliography{ref}

\end{document}